\documentclass[11pt]{amsart}

\usepackage{a4wide}
\usepackage[sort,nospace,compress]{cite}
\usepackage{newtxtext}
\usepackage{newtxmath}
\usepackage{mathrsfs} 
\usepackage{amssymb} 
\usepackage{tikz} 
\usepackage{tikz-cd}
\usepackage{xspace} 
\usepackage{enumitem}
\usepackage{microtype}

\newcommand{\bQ}{\vmathbb{Q}}
\newcommand{\bN}{\vmathbb{N}}
\newcommand{\uf}{\underline{f}}

\newcommand{\bA}{\mathbf{A}} 
\newcommand{\bB}{\mathbf{B}}
\newcommand{\bU}{\mathbf{U}}
\newcommand{\bV}{\mathbf{V}}
\newcommand{\bW}{\mathbf{W}}
\newcommand{\bM}{\mathbb{M}}
\newcommand{\bT}{\mathbf{T}}
\newcommand{\bC}{\mathbf{C}}
\newcommand{\bD}{\mathbf{D}}

\newcommand{\ta}{{\bar{a}}}
\newcommand{\tb}{{\bar{b}}}
\newcommand{\tc}{{\bar{c}}}
\newcommand{\tu}{\bar{u}}
\newcommand{\tv}{\bar{v}}
\newcommand{\tx}{\bar{x}}
\newcommand{\ty}{\bar{y}}

\newcommand{\JEP}{\ensuremath{\operatorname{JEP}}\xspace}
\newcommand{\HP}{\ensuremath{\operatorname{HP}}\xspace}
\newcommand{\AP}{\ensuremath{\operatorname{AP}}\xspace}
\newcommand{\HAP}{\ensuremath{\operatorname{HAP}}\xspace}
\newcommand{\AEP}{\ensuremath{\operatorname{AEP}}\xspace}
\newcommand{\AEPn}{\ensuremath{\AEP^n}\xspace}
\newcommand{\HAPn}{\ensuremath{\HAP^n}\xspace}

\newcommand{\clone}{\mathfrak}
\newcommand{\monoid}{\mathfrak}
\newcommand{\group}{\mathfrak}
\newcommand{\class}{\mathcal}
\newcommand{\cover}{\mathcal}
\newcommand{\cat}{\mathscr}
\newcommand{\algebra}[1]{\vvmathbb{#1}}

\newcommand{\Met}{\operatorname{Met}}
\newcommand{\cMet}{\operatorname{cMet}}
\newcommand{\Age}{\operatorname{Age}}
\newcommand{\End}{\operatorname{End}}
\newcommand{\Aut}{\operatorname{Aut}}
\newcommand{\Pol}{\operatorname{Pol}}
\newcommand{\ar}{\operatorname{ar}}
\newcommand{\Tpp}{\operatorname{Tpp}} 

\newcommand{\Fraisse}{Fra\"\i{}ss\'e\xspace}

\tikzcdset{arrow style=tikz,diagrams={>={Straight Barb[scale=0.8]}}}
\newcommand{\injto}{\hookrightarrow}
\newcommand{\uSigma}{{\underline{\Sigma}}}
\newcommand{\komma}{\mathrel{\downarrow}}
\newcommand{\restr}{\mathord{\upharpoonright}}
\newcommand{\pushout}{\arrow[ur, phantom, "\urcorner", very near start]}
\tikzcdset{every diagram/.append style = {sep=1.5em}}

\newenvironment{enumeratewithhead}[1]{%
  \begin{trivlist}
  \item \textbf{#1}
  \begin{enumerate}
}{%
  \end{enumerate}
  \end{trivlist}
}

\theoremstyle{plain} 
\newtheorem{theorem}{Theorem}[section]
\newtheorem{proposition}[theorem]{Proposition}
\newtheorem{lemma}[theorem]{Lemma}
\newtheorem{corollary}[theorem]{Corollary} 
\theoremstyle{definition}
\newtheorem{definition}[theorem]{Definition}
\newtheorem{example}[theorem]{Example}
\newtheorem{observation}[theorem]{Observation}
\newtheorem*{construction}{Construction}
\theoremstyle{remark}
\newtheorem*{remark}{Remark}

\author[Ch.\,Pech]{Christian Pech}

\address{}
\email{cpech@freenet.de}
\urladdr{https://www.researchgate.net/profile/Christian\_Pech2}

\author[M.\,Pech]{Maja Pech} 
\address{Department of Mathematics\\University of Novi Sad} 
\email{maja@dmi.uns.ac.rs}
\urladdr{http://people.dmi.uns.ac.rs/~maja/}

\thanks{The first author received funding from the European Research Council under the European Community's Seventh Framework Programme (FP7/2007-2013 Grant Agreement no. 257039).\\
The second author was supported by the Ministry of Education and Science of the Republic of Serbia through Grant No.174018.}

\title[Polymorphism clones of homogeneous structures]{Polymorphism clones of homogeneous structures\\[2mm]{\scriptsize Universal Homogeneous Polymorphisms and Automatic Homeomorphicity}}

\begin{document}
\begin{abstract}
	Every clone of functions comes naturally equipped with a topology---the topology of pointwise convergence. A clone $\clone{C}$ is said to have automatic homeomorphicity with respect to a class  $\class{C}$ of clones, if every clone-isomorphism of $\clone{C}$ to a member of $\class{C}$ is already a homeomorphism (with respect to the topology of pointwise convergence). In this paper we study automatic homeomorphicity-properties for polymorphism clones of countable homogeneous relational structures. To this end we introduce and utilize universal homogeneous polymorphisms. Next to two generic criteria for the automatic homeomorphicity of the polymorphism clones of free homogeneous structures we show that the polymorphism clone of the generic poset with reflexive  ordering has automatic homeomorphicity and that the polymorphism clone of the generic poset with strict ordering has automatic homeomorphicity with respect to countable $\omega$-categorical structures. Our results extend and generalize previous results by Bodirsky, Pinsker, and Pongr\'acz.
\end{abstract}
\maketitle
\section{Introduction}
A relational structure is called homogeneous if every isomorphism between finite substructures extends to an automorphism. Homogeneous structures play an important role in model theory because of their close relation to structures whose elementary theory admits quantifier elimination. Also, homogeneous structures form a major source of $\omega$-categorical structures.  

A clone is a set of finitary functions on a given base set that contains all projections and that is closed with respect to composition. Every concrete clone comes equipped with a canonical topology---the topology of pointwise convergence. It was shown by Bodirsky and Pinsker in \cite{BodPin15} that the polymorphism clone of an $\omega$-categorical structure determines this structure up to positive primitive bi-interpretability.
In this paper the authors  asked, which properties of an $\omega$-categorical structure are encoded in its polymorphism clone, considered as an abstract clone. In particular the question is, when can the  canonical topology of the polymorphism clone of a structure be reconstructed from its underlying abstract clone. First steps to find reasonably general  conditions were undertaken by Bodirsky, Pinsker and Pongr\'acz in \cite{BodPinPon17}. Our paper is build on their findings.

What is meant by ``reconstructing the canonical topology of a clone''?  There are several ways to give concrete meaning to the phrase: For a class $\class{K}$ of clones and a clone $\clone{C}\in\class{K}$ we may say that
\begin{enumerate}
  \item  $\clone{C}$ has \emph{reconstruction with respect to $\class{K}$} if whenever $\clone{C}$ is isomorphic to some clone $\clone{D}\in\class{K}$ (as an abstract clone), then there exists already an isomorphism between $\clone{C}$ and $\clone{D}$ that is a homeomorphism (with respect to the canonical topologies of $\clone{C}$ and $\clone{D}$, respectively), or    
  \item $\clone{C}$ has \emph{automatic homeomorphicity with respect $\class{K}$} if whenever $\clone{C}$ is isomorphic to some clone $\clone{D}\in\class{K}$ (as an abstract clone), then \emph{every} isomorphism between $\clone{C}$ and $\clone{D}$ is a homeomorphism.
\end{enumerate}
In this paper we are going to study the second (stronger) option. Note that automatic homeomorphicity is already a non-trivial concept if the class $\class{K}$ consists only of $\clone{C}$. In this case it says that every automorphism of $\clone{C}$ is an autohomeomorphism. 

It should be mentioned that our approach to automatic homeomorphicity is not that of a craftsman but of an engineer. That is, our goal is not, for every given homogeneous structure in question to find the shortest and most elegant proof that its polymorphism clone has automatic homeomorphicity. Rather it is our ambition to find methods as general as possible to show automatic homeomorphicity of the polymorphism clones of whole classes of structures at once. We do so by refining and industrializing the gate techniques that were introduced in \cite{BodPinPon17}. In particular:
\begin{enumerate}
  \item we introduce the notion of strong gate coverings,
  \item we show, how strong gate coverings can be used for showing automatic homeomorphicity of clones,
  \item we introduce the notion of universal homogeneous polymorphisms,
  \item we show that the existence of universal homogeneous polymorphisms of all finite arities for a relational structure implies that its polymorphism clone has a strong gate covering,
  \item we characterize all homogeneous structures that posses universal homogeneous polymorphisms of all finite arities by a property of their age,
\end{enumerate}
Thus we end up with a sufficient condition for the existence of strong gate coverings for polymorphism clones of homogeneous structures. In particular, we show the existence of strong gate coverings for the polymorphism clones of the following structures:
\begin{itemize}
  \item free homogeneous structures whose age has the homo-amalgamation property and is closed with respect to finite products,
  \item  the generic poset. 
\end{itemize}

The paper continues  with new  criteria for the automatic homeomorphicity of clones. In particular we show that the polymorphism clone of a free homogeneous structure $\bU$ has automatic homeomorphicity if 
\begin{enumerate}[label=(\roman*)]
 \item $\Age(\bU)$ has the homo-amalgamation property, 
 \item $\Age(\bU)$ is closed with respect to finite products,
 \item\label{condiii} all constant functions on $U$ are endomorphisms of $\bU$.
 \end{enumerate}
 Moreover, we show that in the above criterion condition \ref{condiii} can be replaced by the following two conditions:
 \begin{enumerate}
 	\item[(iii.a)] $\Aut(\bU)$ acts transitively on $U$,
 	\item[(iii.b)] $\overline{\Aut(\bU)}$ has automatic homeomorphicity.
 \end{enumerate}
Finally, we present a result on automatic homeomorphicity for two non-free homogeneous structures. In particular we show that the polymorphism clone of the generic poset with reflexive order-relation has automatic homeomorphicity and that the polymorphism clone of the generic poset with strict order-relation has automatic homeomorphicity with respect to the class of countable $\omega$-categorical structures.

Some words about the techniques employed by us. For the part about universal homogeneous polymorphisms we use \emph{axiomatic \Fraisse theory}. This is a version of \Fraisse theory, introduced by Droste and G\"obel in \cite{DroGoe92}, that completely abstracts from structures. It is formalized in the language of category theory and encompasses model theoretic  \Fraisse-theory  (including, e.g., Hrushovski's construction and Solecki's projective \Fraisse-limits). The theory has meanwhile been applied, developed, and extended in several works, including \cite{Ros97,Kir09,PecPec12,PecPec13a,Kub15,Car14,Kub14,PecPec16}. We build upon the results from \cite{PecPec13a} on universal homogeneous objects in comma-categories and extend them, in order to obtain our characterization of the existence of universal homogeneous polymorphisms for homogeneous structures. 

Another important tool in our research has been a topological version of Birkhoff's theorem due to Bodirsky and Pinsker \cite{BodPin15} in a rather surprising combination with results about polymorphism homogeneous structures and retracts of \Fraisse-limits (cf.\ \cite{PecPec15,PecPec13a}).

\section{Preliminaries}
\subsection{Clones}
Let $A$ be a set. For $n\in\bN\setminus\{0\}$ we define 
\[
    \clone{O}_A^{(n)}:=\{f\mid f\colon   A^n\to A\},\text{ and} \quad \clone{O}_A := \bigcup_{n\in\bN\setminus\{0\}} \clone{O}_A^{(n)}.
\]
In general, for a set $C\subseteq \clone{O}_A$ we will write $C^{(n)}$ for the set of all $n$-ary functions from $C$. 
We distinguish certain functions in $\clone{O}_A$---the \emph{projections}: For $n\in\bN\setminus\{0\}$, and for $i\in\{1,\dots,n\}$  the projection $e_i^n\in\clone{O}_A^{(n)}$ is defined by 
$e_i^n \colon   (x_1,\dots,x_n)\mapsto x_i$.
Further we define the set of all projections on $A$: $\clone{J}_A:=\big\{e_i^n\mid n\in\bN\setminus\{0\},\, i\in\{1,\dots,n\}\big\}$.
For all $n,m\in\bN\setminus\{0\}$, whenever $f\in \clone{O}_A^{(n)}$, and $g_1,\dots,g_n\in\clone{O}_A^{(m)}$, then the \emph{composition} $f\circ\langle g_1,\dots,g_n\rangle\in\clone{O}_A^{(m)}$ is defined according to 
\[ f\circ\langle g_1,\dots,g_n\rangle\colon   (x_1,\dots,x_m)\mapsto f\big(g_1(x_1,\dots,x_m),\dots,g_n(x_1,\dots,x_m)\big).
\]
\begin{definition}
 A set $\clone{C}\subseteq \clone{O}_A$ is called \emph{clone} on $A$ if $\clone{J}_A\subseteq\clone{C}$, and if 
    $\clone{C}$ is closed with respect to composition.
\end{definition}
Clearly, both, $\clone{O}_A$ and $\clone{J}_A$ are clones. 
If $\clone{C}$ and $\clone{D}$ are clones on $A$, and if $\clone{C}\subseteq\clone{D}$, then we call $\clone{C}$ a \emph{subclone} of $\clone{D}$, and we denote this fact by $\clone{C}\le\clone{D}$. 
\begin{definition}
    Let $A$, $B$ be sets and let $\clone{C}\le \clone{O}_A$, $\clone{D}\le \clone{O}_B$. A function $h\colon  \clone{C}\to\clone{D}$ is called a \emph{clone-homomorphism} if
    \begin{enumerate}
    \item for all $n\in\bN\setminus\{0\}$  we have $h(\clone{C}^{(n)})\subseteq\clone{D}^{(n)}$,
    \item for all $n\in\bN\setminus\{0\}$ and for all $i\in\{1,\dots,n\}$ we have $h(e_i^n) = e_i^n$,
    \item for all $n,m\in\bN\setminus\{0\}$, for all $f\in\clone{C}^{(n)}$, and for all $g_1,\dots,g_n\in\clone{C}^{(m)}$ we have
    \[ h(f\circ\langle g_1,\dots,g_n\rangle) = h(f)\circ\langle h(g_1),\dots,h(g_n)\rangle.\]
    \end{enumerate} 
    A bijective clone-homomorphism will be called \emph{clone-isomorphism}.
\end{definition}

\subsection{The Tychonoff topology on clones}
Let $U$ be a set and let
$n\in\bN\setminus\{0\}$. For every finite subset $M$ of $U^n$ and
every $h\colon   M\to U$ define
$\Phi_{M,h}:=\{f\colon   U^n\to U\mid f\restr_M=h\}$.
Then all the sets of this shape form the basis of a topology on $\clone{O}_U^{(n)}$---the Tychonoff topology (aka the topology of pointwise convergence; here $U$ is considered to be equipped with the discrete topology). With this observation we may consider $\clone{O}_U$ as a topological sum
\[ \clone{O}_U = \bigsqcup_{n\in\bN\setminus\{0\}}\clone{O}_U^{(n)}.\]
Moreover, every clone $\clone{C}\le \clone{O}_U$ may be equipped with the subspace topology with respect to the topology on $\clone{O}_U$. This topology will be called the \emph{canonical topology} of $\clone{C}$. From now on, every clone will implicitly be considered to be equipped with its canonical topology.  
\begin{remark}
    Transformation monoids and permutation groups on $U$ are subsets of $\clone{O}_U^{(1)}$. Thus, they may be equipped with a subspace topology of $\clone{O}_U^{(1)}$. As for clones, in the sequel we will consider every transformation monoid and every permutation group on $U$ to be equipped with this topology, and we will call it the canonical topology of the respective transformation monoid or the permutation group.   
\end{remark}

If $U$ is countably infinite, then, since the space $\clone{O}_U^{(n)}$ is the countable power of a countable discrete space, the above given topology is completely metrizable by an ultrametric. In order to do so we consider $U^n$ as an $\omega$-indexed family $(\tu_i)_{i<\omega}$. Now we consider the function
\begin{gather*}
  D_U^{(n)}\colon  \clone{O}_U^{(n)}\times \clone{O}_U^{(n)}\to\omega^+ \qquad  (f,g)\mapsto 
  \begin{cases} \min\{i\in\omega\mid f(\tu_i)\neq g(\tu_i)\} & \text{if }f\neq g\\
    \omega & \text{if }f=g.
  \end{cases}
\end{gather*}
Now, for $f,g\in \clone{O}_U^{(n)}$, the distance in the mentioned ultrametric is given by
\[
d_U^{(n)}(f,g):= \begin{cases} 2^{-D_U^{(n)}(f,g)} & \text{if }f\neq g\\
  0 & \text{if }f=g.
\end{cases}
\]
Finally, the ultrametrics $d_U^{(n)}$ may be combined to one ultrametric $d_U$ on $\clone{O}_U$ according to
\begin{equation}\label{ultrametric}
d_U(f,g):= \begin{cases}
 1 & \text{if } f\in\clone{O}_U^{(n)}, g\in\clone{O}_U^{(m)}, n\neq m\\
d_U^{(n)}(f,g) & \text{if } f,g\in\clone{O}_U^{(n)}. 
 \end{cases}
\end{equation}
At this point it is important to note that the metric space $(\clone{O}_U,d_U)$ is complete no matter how the enumerations of the $\clone{O}_U^{(n)}$ for $n\in\bN\setminus\{0\}$ are chosen. In particular, if we choose other enumerations of the $\clone{O}_U^{(n)}$, and obtain an ultrametric, say, $d_U'$ on $\clone{O}_U$, then a sequence in $\clone{O}_U$ is going to be a Cauchy-sequence with respect to $d_U$ if and only if it is a Cauchy-sequence with respect to $d_U'$.  In the sequel, for any countable set $U$, we are going to consider $\clone{O}_U$ to be equipped with an ultrametric $d_U$, defined like in \eqref{ultrametric} through arbitrary enumerations of the  $\clone{O}_U^{(n)}$. Moreover, we will consider all subspaces of $\clone{O}_U^{(n)}$ to be equipped with the corresponding restriction of $d_U$, and we will (abusing notation) again denote the restriction by $d_U$.  
    
\subsection{Relational structures}
A \emph{relational signature} is a pair $\underline\Sigma=(\Sigma,\ar)$ where $\Sigma$ is a set of \emph{relational symbols} and $\ar\colon   \Sigma\to\bN\setminus\{0\}$ assigns to each  relational symbol its \emph{arity}. The set of all $n$-ary relational symbols in $\Sigma$ will be denoted by $\Sigma^{(n)}$.

A \emph{$\uSigma$-structure} $\bA$ is a pair $(A,(\varrho^\bA)_{\varrho\in\Sigma})$, such that $A$ is a set, and such that for each $\varrho\in\Sigma$ we have that $\varrho^\bA$ is a relation of arity $\ar(\varrho)$ on $A$.  The set $A$ will be called the \emph{carrier} of $\bA$ and the relations $\varrho^\bA$ will be called the \emph{basic relations} of $\bA$. If the signature $\uSigma$ is of no importance, we will speak only about relational structures. The carriers of a $\uSigma$-structures $\bA,\bB,\bC,\dots$ will usually be denoted by $A,B,C,\dots$, respectively. Moreover, the basic relations of $\bA,\bB,\bC,\dots$ will be denoted by $\varrho^\bA,\varrho^\bB,\varrho^\bC,\dots$, respectively, for each $\varrho\in\Sigma$.

Let $\bA$ and $\bB$ be $\uSigma$-structures. A function $h\colon   A\to B$ is called a \emph{homomorphism} from $\bA$ to $\bB$ if for all $n\in\bN\setminus\{0\}$, for all $\varrho\in\Sigma^{(n)}$ and for all $\ta=(a_1,\dots,a_n)\in \varrho^\bA$ we have that $h(\ta):=(h(a_1),\dots,h(a_n))\in\varrho^\bB$. A function $h\colon   A\to B$ is called \emph{embedding} if $h$ is injective and if for all $n\in\bN\setminus\{0\}$, for all $\varrho\in\Sigma^{(n)}$ and for all $\ta\in A^n$ we have $\ta\in\varrho^\bA \iff h(\ta)\in\varrho^\bB$.
Surjective embeddings are called \emph{isomorphisms}. As usual, isomorphisms of a relational structure $\bA$ onto itself are called \emph{automorphisms}, and homomorphisms of $\bA$ to itself are called \emph{endomorphisms}.  The automorphism group  and the  endomorphism monoid  of $\bA$ will be denoted by $\Aut(\bA)$ and  $\End(\bA)$, respectively. 

Whenever we write $h\colon   \bA\to\bB$, we mean that $h$ is a homomorphism from $\bA$ to $\bB$. Moreover, with $h\colon   \bA\injto\bB$ we denote the fact that $h$ is an embedding from $\bA$ into $\bB$, and we write just $\bA\injto\bB$ if there exists an embedding of $\bA$ into $\bB$.  

Let $\bA$ be a relational structure. For $n\in\bN\setminus\{0\}$, a homomorphism $h\colon  \bA^n\to \bA$ is called an \emph{$n$-ary polymorphism} of $\bA$. With $\Pol^{(n)}(\bA)$ we will denote the set of all $n$-ary polymorphisms of $\bA$. Moreover, we define
\[ \Pol(\bA):=\bigcup_{n\in\bN\setminus\{0\}}\Pol^{(n)}(\bA).\]
It is easy to see, that for every relational structure $\bA$ we have that $\Pol(\bA)$ is a closed subclone  of $\clone{O}_A$---the polymorphism clone of $\bA$. It is less obvious, that every closed subclone on $\clone{O}_A$ may be obtained as the polymorphism clone of a suitable relational structure on $A$ (cf.\ \cite[Lemma 3.1]{BakPix75},\cite[Theorem 1]{Rom77}, \cite[Theorem 4.1]{Poe80}). 

\subsection{Homogeneous structures}
The age of a $\uSigma$-structure $\bU$ is the class of finite $\uSigma$-structures embeddable into $\bU$. It will be denoted by $\Age(\bU)$. A structure $\bA$ is called \emph{younger} than $\bU$ if $\Age(\bA)\subseteq\Age(\bU)$.  According to a classical result by \Fraisse, a class $\class{C}$ of finite $\uSigma$-structures is the age of a countable $\uSigma$-structure if and only if  
\begin{enumerate}
    \item $\class{C}$ has the \emph{hereditary property} (\HP), i.e.
     $\forall \bA,\bB: (\bB\in\class{C}) \land (\bA\injto\bB)\Rightarrow (\bA\in\class{C})$,
    \item $\class{C}$ has the \emph{joint embedding property} (\JEP), i.e.
     $\forall\bA,\bB\in\class{C}\;\exists\bC\in\class{C} : (\bA\injto\bC)\land(\bB\injto\bC)$, 
    \item up to isomorphism, $\class{C}$ contains only countably many structures.
\end{enumerate}
 Thus it is natural to call a class $\class{C}$ of finite $\uSigma$-structures with these three properties an \emph{age}.
 If $\class{C}$ is an age, then by $\overline{\class{C}}$ we will denote the class of all countable structures whose age is contained in $\class{C}$.  

\begin{definition}
    A countable $\uSigma$-structure $\bA$ is called \emph{universal} if every structure from $\overline{\Age(\bA)}$ can be embedded into $\bA$. It is called \emph{homogeneous} if for every $\bB\in\Age(\bA)$ and for all embeddings $\iota_1,\iota_2\colon   \bB\injto\bA$ there exists an automorphism $h$ of $\bA$ such that $\iota_2=h\circ \iota_1$. 
\end{definition}

\begin{definition}
    Let $\class{C}$ be a class of $\uSigma$-structures. We say that $\class{C}$ has the \emph{amalgamation property} (\AP) if for all $\bA$, $\bB$, $\bC$  from $\class{C}$ and for all embeddings $f\colon   \bA\injto \bB$, $g\colon   \bA\injto \bC$, there exists $\bD\in\class{C}$  and embeddings $\hat{f}\colon   \bC\injto \bD$, $\hat{g}\colon   \bB\injto \bD$ such that     the following  diagram commutes:
  \[
  \begin{tikzcd}
    \bC \rar[hook,dashed]{\hat{f}}& \bD\\
    \bA  \uar[hook]{g}\rar[hook]{f}&  \bB\uar[hook,dashed]{\hat{g}}.
  \end{tikzcd}
  \]
\end{definition}
Let us recall the well-known characterization of ages of countable homogenous structures by \Fraisse:
\begin{theorem}[\Fraisse{} \cite{Fra53}]
    Let $\class{C}$ be an age. Then $\class{C}$ is the age of a countable homogeneous structure if and only if it has the \AP. Moreover, any two countable homogeneous structures with the same age are isomorphic. 
\end{theorem}
An age is called a \emph{\Fraisse-class} if it has the \AP. A countable homogeneous $\uSigma$-structure $\bU$ is called a \emph{\Fraisse-limit} of its age $\Age(\bU)$. 
\begin{example}
    Some examples of \Fraisse-classes include the class of finite simple graphs, the class of finite posets (strictly or non-strictly ordered), the class of finite linear orders (strictly or non-strictly ordered), and the class of finite tournaments. 
    The corresponding \Fraisse-limits are the Rado graph (aka the countable random graph, aka the Erd\H{o}s-R\'enyi graph), the countable generic poset, the rationals, and the countable generic tournament, respectively.  
\end{example}
In the following, let $\uSigma$ be a relational signature and let $\cat{C}_\uSigma$ be the category of all $\uSigma$-structures with homomorphisms as morphisms. In $\cat{C}_\uSigma$, the amalgamated free sum is constructed as follows:
\begin{construction}
  Let $\bA$, $\bB_1$, $\bB_2$ be $\uSigma$-structures, such that $\bA\le\bB_1$,  $\bA\le\bB_2$, and such that $B_1\cap B_2=A$. Define $C:= B_1\cup B_2$, and for each $\varrho\in\Sigma$ define $\varrho^\bC:=\varrho^{\bB_1}\cup\varrho^{\bB_2}$, 
  and finally $\bC:=(C,(\varrho^\bC)_{\varrho\in\Sigma})$. Then $\bC$ is  called the \emph{amalgamated free sum} of $\bB_1$ and $\bB_2$ with respect to $\bA$. It is going to be denoted by $\bB_1\oplus_\bA\bB_2$. Note that the  following is always a pushout square in $\cat{C}_\uSigma$:
  \[ 
  \begin{tikzcd}
    \bB_1 \rar[hook]{=}& \bB_1\oplus_\bA\bB_2 \\
    \bA\pushout  \uar[hook]{=}\rar[hook]{=}&\bB_2\uar[hook]{=}.
  \end{tikzcd}
  \]
\end{construction}

\begin{definition}
  We say, that the age of a $\uSigma$ structure $\bU$ has the \emph{free amalgamation property} if $\Age(\bU)$ is closed with respect to amalgamated free sums in $\cat{C}_\uSigma$.
\end{definition}

\section{Automatic homeomorphicity}
\begin{definition}
    Let $\class{K}$ be a class of structures (possibly over different signatures), and let $\bU\in\class{K}$. We say that 
    \begin{itemize}
        \item $\Aut(\bU)$ has \emph{automatic homeomorphicity} with respect to $\class{K}$ if every group-isomorphism from $\Aut(\bU)$ to the automorphism group of a member of $\class{K}$ is a homeomorphism,
        \item $\overline{\Aut(\bU)}$ has \emph{automatic homeomorphicity} with respect to $\class{K}$ if every monoid-isomorphism from $\overline{\Aut(\bU)}$ to a closed submonoid of $\End(\bV)$  is a homeomorphism, for every $\bV\in\class{K}$,
        \item $\End(\bU)$ has \emph{automatic homeomorphicity} with respect to $\class{K}$ if every monoid-isomorphism from $\End(\bU)$ to the endomorphism monoid of a member of $\class{K}$ is a homeomorphism,
        \item $\Pol(\bU)$ has \emph{automatic homeomorphicity} with respect to $\class{K}$ if every clone-isomorphism from $\Pol(\bU)$ to the polymorphism clone of a member of $\class{K}$ is a homeomorphism. 
    \end{itemize}
    The phrase ``with respect to $\class{K}$'' will be dropped whenever $\class{K}$ consists of all structures on $U$. 
\end{definition}

The notion of automatic homeomorphicity for transformation semigroups and for clones was introduced by Bodirsky, Pinsker and Pongr\'acz in \cite{BodPinPon17}.  They proved automatic homeomorphicity for the following clones:
\begin{itemize}
    \item the Horn clone (i.e., the smallest closed clone containing all injective functions from $\clone{O}_\omega$),
    \item the closed subclones of $\clone{O}_\omega$ that contain $\clone{O}^{(1)}$,
    \item the polymorphism clone of the Rado graph,
    \item the clone of essentially injective polymorphisms of the Rado-graph,
    \item the $17$ minimal tractable clones over the Rado graph (cf. \cite{BodPin11}). 
\end{itemize}
Recently this list was expanded by Behrisch, Truss, and Vargas-Garc\'\i{}a in \cite{BehTruVar17}, \cite{TruVar16} to include the following clones:
\begin{itemize}
	\item the clone generated by $\End(\bU)$, where $\bU$ is a countable structure such that $\End(\bU)$ has automatic homeomorphicity,
	\item $\Pol(\bQ,\le)$,
	\item $\Pol(\bQ,\operatorname{betw})$, where $\operatorname{betw}(x,y,z)\equiv x\le y\le z \lor z\le y\le x$,
	\item $\Pol(\bQ,\operatorname{circ})$, where $\operatorname{circ}(x,y,z)\equiv x\le y\le z \lor y\le z\le x \lor z\le x\le y$,
	\item $\Pol(\bQ,\operatorname{sep})$, where $\operatorname{sep}(x,y,z,t)\equiv\operatorname{circ}(x,y,z)\lor \operatorname{circ}(x,t,y) \lor \operatorname{circ}(x,z,y) \lor \operatorname{circ}(x,y,t)$.
\end{itemize}

To show automatic homeomorphicity for the polymorphism clone of a countable homogeneous structure $\bU$ with respect to a class $\class{K}$ of structures, Bodirsky, Pinsker and Pongr\'acz in \cite{BodPinPon17} devised the following programme:
\begin{enumerate}
    \item\label{stepAut} show that $\Aut(\bU)$ has automatic homeomorphicity with respect to $\class{K}$,
    \item\label{stepEmb} show that $\overline{\Aut(\bU)}$ has automatic homeomorphicity with respect to $\class{K}$,
    \item\label{stepEnd} show that every isomorphism from $\End(\bU)$ to the the endomorphism monoid of a member of $\class{K}$ is continuous,
    \item\label{stepPol} show that every isomorphism from $\Pol(\bU)$ to the the polymorphism clone of a member of $\class{K}$ is continuous,
    \item\label{stepTB} show that every continuous isomorphism from $\Pol(\bU)$ to the polymorphism clone of a member of  $\class{K}$ is a homeomorphism.  
\end{enumerate}
Step \ref{stepAut} of this strategy is outsourced to group theory. To be more precise, there are two standard ways to show automatic homeomorphicity for groups---the small index property (recall that a  structure $\bU$ is said to have the \emph{small index property} if every subgroup of $\Aut(\bU)$ of index $<2^{\aleph_0}$ is open in $\Aut(\bU)$)  \cite{DixNeuTho86,Tru89,Hru92,HodHodLasShe93,Her98,Sol05,KecRos07}), and Rubin's (weak) $\forall\exists$-interpretations (cf.\ \cite{Rub94,BarPhD}). If  $\bU$ has the small index property, then $\Aut(\bU)$ has automatic homeomorphicity. On the other hand, if $\bU$ has a weak $\forall\exists$-interpretation, then $\Aut(\bU)$ has automatic homeomorphicity with respect to the class of countable $\omega$-categorical structures.

Step \ref{stepEmb} bases on the following observation:
\begin{proposition}[{\cite[Lemma 12]{BodPinPon17}}]\label{embcriterion}
	 If a closed transformation monoid $\monoid{M}$ on a countable set has a dense group $\group{G}$ of units, and if among the injective endomorphisms of $\monoid{M}$ only the identical endomorphism  fixes all elements of $\group{G}$ point-wise, then from the automatic homeomorphicity of $\group{G}$ with respect to $\class{K}$ follows the automatic homeomorphicity  of $\monoid{M}$ with respect to $\class{K}$.
\end{proposition}	
 It is shown in \cite[Theorem 21]{BodPinPon17} that this criterion applies to the monoid of self-embeddings of a countable homogeneous structure $\bU$ whenever $\Aut(\bU)$ has automatic homeomorphicity with respect to $\class{K}$, no algebraicity, and whenever $\bU$ has the \emph{joint extension property} (cf.\ \cite[Definition 18]{BodPinPon17}).

Step \ref{stepEnd}  relies on a so called \emph{gate technique}:
\begin{definition}[{\cite[Definition 3.1]{PecPec16}, implicit in \cite{BodPinPon17}}]\label{MonoidStrongGateCovering}
    Given a transformation monoid $\monoid{M}$ on a countably infinite set $A$. Let $\group{G}$ be the group of units in $\monoid{M}$, and let $\overline{\group{G}}$ be the closure of $\group{G}$ in $\monoid{M}$. Then we say that $\monoid{M}$  has a \emph{gate covering} if there exists an open covering $\cover{U}$ of $\monoid{M}$ and elements $f_U\in U$, for every $U\in\cover{U}$, such that for all $U\in\cover{U}$ and for all Cauchy-sequences $(g_n)_{n\in\bN}$ of elements from $U$  there exist Cauchy-sequences $(\kappa_n)_{n\in\bN}$ and $(\iota_n)_{n\in\bN}$ of elements from $\overline{\group{G}}$ such that for all $n\in\bN$ we have
            $g_n = \kappa_n\circ f_U\circ\iota_n$.
    
\end{definition}
Now Step \ref{stepEnd} can be fulfilled by observing that if $\overline{\Aut(\bU)}$ has automatic homeomorphicity with respect to $\class{K}$ and if $\End(\bU)$ has a gate covering, then every isomorphism from $\End(\bU)$ to the endomorphism monoid of a member of $\class{K}$ is continuous.  

Another gate-technique may be used to fulfill Step \ref{stepPol}:
\begin{definition}[{\cite[Definition 36]{BodPinPon17}}]\label{GateCovering}
    Let $\clone{C}$ be a clone.  Then $\clone{C}$  is said to have a \emph{gate covering} if there exists an open covering $\cover{U}$ of $\clone{C}$ and functions $f_U\in U$, for every $U\in\cover{U}$, such that for each $U\in\cover{U}$ and for all Cauchy-sequences $(g_n)_{n\in\bN}$ of functions from $U$ (all of the same arity $k$) there exist Cauchy-sequences $(\kappa_n)_{n\in\bN}$ and $(\iota^i_n)_{n\in\bN}$ ($i=1,\dots,k$) of functions from $\clone{C}^{(1)}$ such that 
    \[
        g_n(x_1,\dots,x_k) = \kappa_n(f_U(\iota^1_n(x_1),\dots,\iota^k_n(x_k))).
    \]
\end{definition}
In \cite[Theorem 38]{BodPinPon17} it is shown that whenever $\Pol(\bU)$ has a gate covering then every isomorphism from $\Pol(\bU)$ to the polymorphism clone of a member of $\class{K}$, whose restriction to $\End(\bU)$ is continuous, is itself continuous.

Finally, in Step~\ref{stepTB} a topological version of Birkhoff's theorem from \cite{BodPin15} is used to show that every continuous isomorphism from $\Pol(\bU)$ to the polymorphism clone of some structure from $\class{K}$ is open, too. 

The above sketched strategy was used in \cite{BodPinPon17} for showing automatic homeomorphicity of the polymorphism clone of the Rado graph.

Each of the 5 steps carries substantial difficulties. In the following we are going to short-circuit this process, by proving automatic homeomorphicity of the polymorphism clone of a structure $\bU$ without showing first the automatic homeomorphicity of $\overline{\Aut(\bU)}$ and/or $\End(\bU)$. 

In particular, we devise two new strategies for showing automatic homeomorphicity for the polymorphism clone of a countable homogeneous structure $\bU$ with respect to a class $\class{K}$ of structures:
  \begin{enumeratewithhead}{First strategy}
		\item Show that every isomorphism from the polymorphism clone of a member of $\class{K}$ to $\Pol(\bU)$ is continuous.
		\item Show that every continuous isomorphism from the polymorphism clone of a member of $\class{K}$ to $\Pol(\bU)$ is a homeomorphism.
	\end{enumeratewithhead}
	\begin{enumeratewithhead}{Second strategy}
		\item Show that $\Aut(\bU)$ has automatic homeomorphicity with respect to $\class{K}$.
		\item 	Show that $\overline{\Aut(\bU)}$ has automatic homeomorphicity with respect to $\class{K}$. 
	\item Show that every isomorphism from $\Pol(\bU)$ to the polymorphism clone of a member of $\class{K}$  is continuous. 
	\item Show that every continuous isomorphism from $\Pol(\bU)$ the the polymorphism clone of another member of $\class{K}$ is a homeomorphism.
	\end{enumeratewithhead}

Both our strategies  base on a gate-technique: The following definition is a slightly stronger  formulation of Definition~\ref{GateCovering} in the spirit of Definition~\ref{MonoidStrongGateCovering}:
\begin{definition}
    Let $\clone{C}$ be a clone, let $\group{G}$ be the group of units in $\clone{C}^{(1)}$, and let $\overline{\group{G}}$ be the closure of $\group{G}$ in $\clone{C}^{(1)}$. Then $\clone{C}$  is said to have a \emph{strong gate covering} if there exists an open covering $\cover{U}$ of $\clone{C}$ and functions $f_U\in U$, for every $U\in\cover{U}$, such that for each $U\in\cover{U}$ and for all Cauchy-sequences $(g_n)_{n\in\bN}$ of functions from $U$ (each of the same arity $k$) there exist Cauchy-sequences $(\kappa_n)_{n\in\bN}$ and $(\iota^i_n)_{n\in\bN}$ ($i=1,\dots,k$) of functions from $\overline{\group{G}}$ such that $g_n(x_1,\dots,x_k) = \kappa_n(f_U(\iota^1_n(x_1),\dots,\iota^k_n(x_k)))$.
\end{definition}
Strong gate coverings allow to lift continuity properties:
\begin{lemma}\label{liftcont}
	Let $\bA$ and $\bB$ be two countable relational structures, such that $\Pol(\bA)$ has a strong gate covering. Let $h\colon \Pol(\bA)\to\Pol(\bB)$ be a clone homomorphism whose restriction to $\overline{\Aut(\bA)}$ is continuous. Then $h$ is continuous, too. 
\end{lemma}
\begin{proof}
Let $(v_n)_{n\in\bN}$ be a Cauchy-sequence of $k$-ary polymorphisms of $\bA$. Since $(\Pol(\bA),d_A)$ is complete, $(v_n)_{n\in\bN}$ is convergent---say to $v\in\Pol^{(k)}(\bA)$. 
       
Let $(\cover{U},(f_U)_{U\in\cover{U}})$ be a strong gate covering of $\Pol(\bA)$. Then there exists a $U\in\cover{U}$ and an $n_0\in\bN$ such that for all $n\ge n_0$ we have $v_n\in U$. Without loss of generality, assume that $n_0=0$. By the definition of strong gate coverings there exist Cauchy-sequences $(\kappa_n)_{n\in\bN}$ and $(\iota^i_n)_{n\in\bN}$ ($i=1,\dots,k$) in $\overline{\Aut(\bA)}$, such that 
$v_n(x_1,\dots,x_k) = \kappa_n(f_U((\iota^1_n(x_1),\dots,\iota^k_n(x_k)))$,
for all $n\in\bN$. In particular, with $\kappa=\lim_{n\to\infty}\kappa_n$ and $\iota^i=\lim_{n\to\infty} \iota^i_n$, we have 
$v(x_1,\dots,x_k) = \kappa(f_U((\iota^1(x_1),\dots,\iota^k(x_k)))$.
Because $h\restr_{\overline{\Aut(\bA)}}$ is continuous,  we have $\lim_{n\to\infty} h(\kappa_n)=h(\kappa)$ and $\lim_{n\to\infty}h(\iota^i_n) = h(\iota^i)$, for all $i=1\dots k$.  

Since $h$ is a clone-isomorphism, we have 
$h(v_n)(x_1,\dots,x_k) = h(\kappa_n)(h(f_U)(h(\iota^1_n)(x_1),\dots,h(\iota^k_n)(x_k)))$.
Thus, since the composition of functions is continuous, we have that the sequence $(h(v_n))_{n\in\bN}$ converges to $h(v)$. From this, it follows that $h$ is continuous. 
\end{proof}

\subsection{About the first strategy}
\begin{proposition}\label{autocontcriterion}
    Let $\bA$ and $\bB$ be two countable relational structures, such that $\Pol(\bB)$ has a strong gate covering.
    Let $h\colon   \Pol(\bA)\to\Pol(\bB)$ be a continuous clone-isomorphism. Then $h$ is a homeomorphism.
\end{proposition}
Before coming to the proof of this proposition, let us make some auxiliary  observations:
\begin{lemma}\label{univcontmonoid}
    Let $A$, $B$ be countable sets, and let $\monoid{M}_1\le \clone{O}_A^{(1)}$, $\monoid{M}_2\le \clone{O}_B^{(1)}$ be monoids, such that $\monoid{M}_1$ has a dense set of units. Let $h\colon   \monoid{M}_1\to \monoid{M}_2$ be a continuous homomorphism. Then $h$ is uniformly continuous from $(\monoid{M}_1, d_A)$ to $(\monoid{M}_2,d_B)$.
\end{lemma}
\begin{proof}
    Suppose that the metrics $d_A$ and $d_B$ are induced by enumerations $\ta$ and $\tb$ of $A$ and $B$, respectively. Let $e_1$, $e_2$ be the neutral elements of $\monoid{M}_1$ and of $\monoid{M}_2$, respectively. Let $\varepsilon>0$.   Since $h$ is continuous at $e_1$, there exists a $\Delta\in\bN\setminus\{0\}$  such that, with   $\delta:=2^{-\Delta}$, for all $m\in \monoid{M}_1$ with $d_\ta(m,e_1)\le \delta$ we have $d_\tb(h(m),e_2)\le\varepsilon$. 
    
    Let $m,m'\in \monoid{M}_1$ with  $d_A(m,m')\le \delta$. Then we have
    $(m(a_0),\dots,m(a_{\Delta-1}))=(m'(a_0),\dots,m'(a_{\Delta-1}))=:\tc$.
    But since the units lie dense in $\monoid{M}_1$, there exists a unit $g\in \monoid{M}_1$ with 
      $(g(a_0),\dots,g(a_{\Delta-1}))=\tc$.
    Consider now $\widetilde{m}:= g^{-1}\circ m$ and $\widetilde{m}':=g^{-1}\circ m'$. Then $d_A(\widetilde{m},e_1)\le\delta$ and $d_A(\widetilde{m}',e_1)\le \delta$. 
    
    Now we compute
    \[
        \varepsilon\ge d_B(h(\widetilde{m}),e_2)= d_B(h(g^{-1}\circ m),e_2)=d_B(h(g)^{-1}\circ h(m),e_2)= d_B(h(m),h(g)) 
    \]
    In the same way we obtain $d_B(h(m'),h(g))\le\varepsilon$. Since $d_B$ is an ultrametric, we finally conclude that $d_B(h(m),h(m'))\le\varepsilon$. 
\end{proof}

We will further need the following basic facts about metric spaces and uniform continuous functions:
\begin{lemma}[Hausdorff {\cite[Page 368]{Hau14}}]\label{uniqueext}
    Let $(M_1,d_1)$ be a metric space and let $(M_2,d_2)$ be a complete metric space. Then every uniformly continuous function $f\colon  (M_1,d_1)\to(M_2,d_2)$ has a unique uniformly continuous extension to the completion of $(M_1,d_1)$.
\end{lemma}
\begin{corollary}
    Let $\Met$ be the category of metric spaces with uniformly continuous functions. Let $\cMet$ be the full subcategory of $\Met$ spanned by all complete metric spaces. Then the assignment that maps every metric space $\bM$ to its completion $\widehat{\bM}$ and that maps every uniform continuous function $f\colon  \bM_1\to\bM_2$ to its unique extension $\hat{f}\colon  \widehat{\bM}_1\to\widehat{\bM}_2$ is a functor from $\Met$ to $\cMet$.  
\end{corollary}
\begin{proof}
This is folklore.
\end{proof}
\begin{remark} 
  In fact, $\cMet$ is a reflective subcategory of $\Met$, and the completion functor is the corresponding reflector. This is one of the earliest examples of reflective subcategories. In Freyd's PhD-thesis  (this is the place where Freyd introduced notion of reflective subcategories)  it is shown that the class of complete metric spaces induces a reflective subcategory in the category of metric spaces with non-expansive mappings (cf.\ \cite[Page 25]{Fre60}). The same proof functions for the situation with uniformly continuous functions (cf.\ \cite[Page 79]{Fre64}).  
\end{remark}
We are going to denote the completion functor by $C$. 
Finally we are going to make use of the following observation by Lascar:
\begin{proposition}[{\cite[Corollary 2.8]{Las91}}]\label{lascarhom}
Let $\bA$ and $\bB$ be countable relational structures and let $f$ be a continuous isomorphism from $\Aut(\bA)$ to $\Aut(\bB)$. Then $f$ is a homeomorphism. 
\end{proposition}
Eventually we can come to the proof of Proposition~\ref{autocontcriterion}:
\begin{proof}[Proof of Proposition~\ref{autocontcriterion}]
Let  $f:= h\restr_{\Aut(\bA)}$. Since $h$ is continuous, we have that $f$ is continuous, too. Thus, by Proposition~\ref{lascarhom}, $f$ is a homeomorphism. By Lemma~\ref{univcontmonoid}, $f\colon  (\Aut(\bA),d_A)\to(\Aut(\bB),d_B)$ and $f^{-1}\colon  (\Aut(\bB),d_B)\to(\Aut(\bA),d_A)$ are uniformly continuous. That is, $f$ is an isomorphism in the category $\Met$. Let $\hat{f}:=C(f)$ be the unique uniformly continuous extension of $f$ to $\overline{\Aut(\bA)}$. Then, since $C$ is a functor, we have that  $\hat{f}\colon  \overline{\Aut(\bA)}\to\overline{\Aut(\bB)}$ is an isomorphism in the category $\cMet$, and in particular we have that $C(f^{-1})=C(f)^{-1}=\hat{f}^{-1}$ holds.

Let now $g:=h\restr_{\overline{\Aut(\bA)}}$. Since $h$ is continuous, it follows that $g\colon  \overline{\Aut(\bA)}\to\overline{\Aut(\bB)}$ is continuous, too. Thus, from Lemma~\ref{univcontmonoid} we conclude that $g\colon  (\overline{\Aut(\bA)},d_A)\to (\overline{\Aut(\bB)},d_B)$ is uniformly continuous. Because, clearly, we have $g\restr_{\Aut(\bA)}=f$, we conclude from Lemma~\ref{uniqueext}, that $g=C(f)=\hat{f}$. Thus $g\colon  \overline{\Aut(\bA)}\to\overline{\Aut(\bB)}$ is a homeomorphism. 

Now, since $h^{-1}$ is a clone-homomorphism, and since $(h^{-1})\restr_{\overline{\Aut(\bB)}} = g^{-1}$, and since $g^{-1}$ is continuous, it follows from Lemma~\ref{liftcont} that $h^{-1}$ is continuous, too.
\end{proof}

\begin{corollary}
  Let $\class{K}$ be a class of structures and let $\bU\in\class{K}$, such that $\Pol(\bU)$  has a strong gate covering. Then $\Pol(\bU)$ has automatic homeomorphicity with respect to $\class{K}$ if and only if every isomorphism from $\Pol(\bU)$ to the polymorphism clone of a member of $\class{K}$ is open.
\end{corollary}
\begin{proof}
  Suppose that every isomorphism from $\Pol(\bU)$ to the polymorphism clone of a member of $\class{K}$ is open. Let $\bV\in\class{K}$, and let $h\colon  \Pol(\bU)\to\Pol(\bV)$ be an isomorphism. Then $h$ is open. Hence $h^{-1}\colon  \Pol(\bV)\to\Pol(\bU)$ is a  continuous clone isomorphism. Since $\Pol(\bU)$ has a strong gate covering, it follows from Proposition~\ref{autocontcriterion}, that $h^{-1}$ is a homeomorphism. Thus, $h$ is a homeomorphism, too.
    
    The proof of the other direction of the claim is trivial.
\end{proof}
In order to fulfill our first strategy, we may use the following results from \cite{BodPinPon17}:
\begin{proposition}[{\cite[Proposition 27]{BodPinPon17}}]\label{bbpopen}
    Let $\bU$ be a relational structure such that $\Pol(\bU)$ contains all constant functions. Then every isomorphism from $\Pol(\bU)$ to another clone of functions is open. 
\end{proposition}
If  it is known that  $\End(\bU)$ has automatic homeomorphicity with respect to $\class{K}$, then there is an alternative way to show openness for the isomorphisms from $\Pol(\bU)$ to polymorphism clones of structures from $\class{K}$, provided $\Aut(\bU)$ acts transitively on $U$:
\begin{proposition}[{\cite[Proposition 33]{BodPinPon17}}]
  If $\Aut(\bU)$ is transitive, then every injective clone homomorphism $h$ from $\Pol(\bU)$ to another clone, whose restriction to $\End(\bU)$ is open, is itself open.
\end{proposition}
\begin{remark}
  Note that our first strategy does not require us to show automatic homeomorphicity of $\Aut(\bU)$, $\overline{\Aut(\bU)}$, or $\End(\bU)$, in order to derive the automatic homeomorphicity of $\Pol(\bU)$.
\end{remark}

\subsection{About the second strategy}
Our second strategy is relatively similar to the one from \cite{BodPinPon17}. The first step remains the same. For the second step, we use a recent result about countable saturated structures: 
\begin{proposition}[{\cite[Proposition 2.5]{PecPec17}}]\label{newcrit}
		Let $\bU$ be a countable saturated structure such that $\Aut(\bU)$ has a trivial center. Then every endomorphism of $\overline{\Aut(\bU)}$ that fixes $\Aut(\bU)$ element-wise, is the identity on $\overline{\Aut(\bU)}$. 
\end{proposition}
\begin{corollary}\label{newmonoidauthom}
	Let $\bU$	be a countable saturated structure such that $\Aut(\bU)$ has a trivial center and such that $\Aut(\bU)$ has automatic homeomorphicity with respect to $\class{K}$. Then $\overline{\Aut(\bU)}$ has automatic homeomorphicity with respect to $\class{K}$, too. 
\end{corollary}
\begin{proof}
	This is a direct consequence of Proposition~\ref{newcrit} and Proposition~\ref{embcriterion} .
\end{proof}
The rest of the second strategy uses, apart from strong gate coverings,  a technique from \cite{BodPinPon17}, that was used there in order to show automatic homeomorphicity of the polymorphism clone of the Rado graph. We are going to make this technique applicable to a much wider class of relational structures. The key  is going to be a topological version of Birkhoff's theorem due to Bodirsky and Pinsker:
\begin{theorem}[{\cite[Theorem 4]{BodPin15}}]
  Let $\algebra{A}$ and $\algebra{B}$ be countable algebras over the same signature, whose clones of term functions are $\clone{A}$ and $\clone{B}$, respectively. Suppose that $\overline{\clone{A}}^{(1)}$ has an oligomorphic group of units and that $\algebra{B}$ is finitely generated. Then the following are equivalent:
  \begin{enumerate}
	\item $\algebra{B}\in \operatorname{H}\operatorname{S}\operatorname{P}^{\operatorname{fin}}(\algebra{A})$,
	\item the clone homomorphism $\xi\colon \clone{A}\to\clone{B}$ that maps $f^{\algebra{A}}$ to $f^{\algebra{B}}$, for all basic operations $f$, exists and is Cauchy-continuous.
  \end{enumerate}
\end{theorem}
\begin{remark}
	Note that if $\clone{A}$ is closed, then Cauchy-continuous may be replaced by just continuous in the previous Theorem.
\end{remark}

Before being able to state the main result of this subsection, another, by now well-established property of ages of relational structures needs to enter the stage---the homo-amalgamation property (\HAP):
\begin{definition}
  Let $\class{C}$ be a class of $\uSigma$-structures. We say that $\class{C}$ has the \emph{homo-amalgamation property} (\HAP) if for all $\bA$, $\bB$, $\bC$ from $\class{C}$, for all homomorphisms $f\colon   \bA\to \bB$, and for all embeddings $g\colon   \bA\injto \bC$, there exists $\bD\in\class{C}$, a homomorphism $\hat{f}\colon   \bC\to \bD$, and an embedding $\hat{g}\colon   \bB\injto \bD$ such that  the following diagram commutes:
  \[
  \begin{tikzcd}
    \bC \rar[dashed]{\hat{f}}& \bD\\
    \bA  \uar[hook]{g}\rar{f}&  \bB\uar[hook,dashed]{\hat{g}}.
  \end{tikzcd}
  \]
\end{definition}
In the rest of this subsection, we are going to prove the following result:
\begin{proposition}\label{autohomhap}
  Let $\bU$ be a countable, homogeneous, $\omega$-categorical relational structure such that 
  \begin{enumerate}
    \item $\Aut(\bU)$ acts transitively on $U$,
    \item $\Age(\bU)$ has the free amalgamation property,
    \item $\Age(\bU)$ is closed with respect to finite products,
    \item $\Age(\bU)$ has the \HAP.
  \end{enumerate}
  Then every continuous isomorphism from $\Pol(\bU)$ to another closed subclone $\clone{D}$ of $\clone{O}_U$ is a homeomorphism.
\end{proposition}
As usual, before proving this proposition, let us collect the necessary tools: Recall that a consistent set of primitive positive formulae with free variables in $\{x_1,\dots,x_n\}$ is called a \emph{primitive positive $n$-type}. To a structure $\bA$ and a relation $\sigma\subseteq A^n$ we may associate a primitive positive type according to 
\[
	\Tpp_\bA(\sigma):=\{\varphi(x_1,\dots,x_n)\mid \forall \ta\in\sigma:\,\bA\models\varphi(\ta)\}.
\]
Primitive positive types that arise in this way are called \emph{closed}. A primitive positive $n$-type $\Psi$ is called \emph{complete} if there exists a structure $\bA$ and a finite relation $\sigma\subseteq A^n$, such that $\Psi=\Tpp_\bA(\sigma)$. 

Recall also that a structure  is called \emph{weakly oligomorphic} if its endomorphism monoid has just finitely many invariant relations of every given finite arity \cite{MasPec11}. By a result by Ma\v{s}ulovi\'c \cite[Theorem 2]{Mas16}, a countable structure $\bA$ is weakly oligomorphic if and only if its polymorphism clone has just finitely many invariant relations of every finite arity (cf.\ also \cite[Proposition 4.8]{PecPec15}). Finally, by \cite[Proposition 4.7]{PecPec15}, $\bA$ is weakly oligomorphic, if and only if it affords just finitely many closed primitive positive types of every finite arity. Note that this implies immediately that in a countable weakly oligomorphic structure all closed primitive positive types are complete.
\begin{lemma}\label{oneGeneratedRel}
    Let $\bA$ be a weakly oligomorphic relational structure with quantifier elimination for primitive positive formulae, whose age is closed with respect to finite products. Then every complete primitive positive type $\Phi$ over $\bA$ is of the shape $\Tpp_\bA(\ta)$ for a suitable tuple $\ta$ of elements of $A$. 
\end{lemma}
\begin{proof}
    Let $\Phi$ be an $m$-ary complete primitive positive type over $\bA$. Then, since $\bA$ is weakly oligomorphic, there exists $\{\ta_1,\dots,\ta_n\}\subseteq A^m$ such that $\Phi=\Tpp_\bA(\{\ta_1,\dots,\ta_n\})$. Suppose $\ta_j=(a_{1,j},\dots,a_{m,j})$ for $j\in\{1,\dots,n\}$. Let $\tb_i:=(a_{i,1},\dots,a_{i,n})$, for $i\in\{1,\dots,m\}$. Let $\bB$ be the substructure of $\bA^n$ spanned by $\{\tb_1,\dots,\tb_m\}$. Since $\Age(\bA)$ is closed with respect to finite products, we have $\bB\in\Age(\bA)$. Let $\iota\colon \bB\injto\bA$ be an embedding from $\bB$ into $\bA$, and let $c_i:=\iota(\tb_i)$, for $i\in\{1,\dots,m\}$. Then $\Tpp_\bA((c_1,\dots,c_n))$ contains the same atomic formulae like $\Phi$. Since $\bA$ has quantifier elimination for primitive positive formulae, we have $\Phi=\Tpp_\bA((c_1,\dots,c_n))$. 
\end{proof}

\begin{proposition}\label{autoopencriterion}
    Let $\bU$ be a countable, homogeneous, $\omega$-categorical relational structure with quantifier elimination for primitive positive formulae such that 
    \begin{enumerate}
    \item $\Aut(\bU)$ acts transitively on $U$,
    \item $\Age(\bU)$ has the free amalgamation property,
    \item $\Age(\bU)$ is closed with respect to finite products.
    \end{enumerate}
     Then every continuous isomorphism to another closed subclone $\clone{D}$ of $\clone{O}_U$ is a homeomorphism.
\end{proposition}
\begin{proof}
    The proof follows the lines of the proof of \cite[Lemma 51]{BodPinPon17}, where our claim is proved for the special case when $\bU$ is the Rado graph. Let $\xi\colon \Pol(\bU)\to\clone{D}$ be a continuous clone-isomorphism.
    
    First, for every $n\in\bN\setminus\{0\}$, and for every $f\in\Pol^{(n)}(\bU)$, let $\underline{f}$ be an $n$-ary operation symbol. Let $\uSigma$ be the algebraic signature, that consists of all newly defined operation symbols. Now we consider the algebras $\algebra{U}=(U,\Pol(\bU))$, $\algebra{D}=(U,\clone{D})$ as $\uSigma$-algebras, where for every $\uf\in\Sigma$ the interpretation of $\uf$ in $\algebra{U}$ is $f$ and the interpretation of $\uf$ in $\algebra{D}$ is  $\xi(f)$.  
    
    Let $\algebra{B}$ be some finitely generated subalgebra of $\algebra{D}$ with at least two elements, and let $r\colon \clone{D}\to\clone{O}_B$ be the restriction homomorphism defined by $r(g):=g\restr_B$. Let $\clone{D}_B$ be the image of $\clone{D}$ under $r$. Then $\algebra{B}=(B,\clone{D}_B)$, where  $\uf\in\Sigma$ is interpreted as $r(\xi(f))$, for all $f\in\Sigma$. 
    
    Since $(B,\clone{D}_B)$ is a subalgebra of $(U,\clone{D})$, it follows from \cite[Proposition 5]{BodPin15} that $r\colon \clone{D}\to\clone{D}_B$ is a continuous clone-homomorphism.  
    
    In the following, we will show that $\xi':=r\circ\xi$ is a homeomorphism. When this is done, it follows that $\xi$ is a homeomorphism, too, since in this case we have that $r$ is bijective, thus $r^{-1}$ is an open clone isomorphism, and thus $\xi=r^{-1}\circ \xi'$ is open.
    
    Since $\xi'$ is a continuous clone-homomorphism, and since $\algebra{B}$ is finitely generated, it follows from the topological Birkhoff theorem that $\algebra{B}$ is contained in the pseudovariety generated by $\algebra{U}$. In other words, $\algebra{B}$  is a homomorphic image of a subalgebra of a finite power of $\algebra{U}$. Let $\algebra{S}$ be the  corresponding subalgebra in this process, and let $\sim$ be the kernel of the surjective homomorphism from $\algebra{S}$ to $\algebra{B}$. Then for some $n$, we have that $S$ is an $n$-ary invariant relation of $\Pol(\bU)$. Since $\bU$ is $\omega$-categorical, it follows from \cite[Theorem 4]{BodNes06}, that $S$ is definable by a set $\Psi$ of primitive positive formulae in the language of $\bU$. We may suppose without loss of generality that $\Psi=\Tpp_\bU(S)$. Also, without loss of generality, we may assume that $\Psi$ does not contain a formula of the shape $x_i=x_j$ for $i\neq j$. Thus, by Lemma \ref{oneGeneratedRel}, $S$ contains at least one irreflexive tuple. 
    
    The relation $\sim$ is a congruence relation of the algebra $\algebra{S}$, i.e., it is invariant under all term-functions of $\algebra{S}$. Note that the term functions of $\algebra{S}$ are just the elements of $\Pol(\bU)$ in their natural action on $n$-tuples. Thus, if we consider $\sigma^\sim:=\{\tu\tv\mid \tu,\tv\in S, \tu\sim\tv\}$, then $\sigma^\sim$ is a $2n$-ary invariant relation of $\Pol(\bU)$. By the same reasoning as above, $\sigma^\sim$ is defined through a set $\Phi$ of primitive positive formulae over $\bU$. Again, we may assume that $\Phi=\Tpp_\bU(\sigma^\sim)$. To improve readability, we use the following convention for the names of the variables in formulae from $\Phi$: Every formula in $\varphi\in\Phi$ shall be of the form $\varphi(\tx,\ty)$, where $\tx=(x_1,\dots,x_n)$ and where $\ty=(y_1,\dots,y_n)$. Clearly, because $\sim$ is reflexive and symmetric, if $\varphi(\tx,\ty)\in\Phi$, then we also have $\varphi(\tx,\tx)\in\Psi$ and $\varphi(\ty,\tx)\in\Phi$. 
    
    Observe that  $\Phi$ does not contain a formula of the shape $x_i=y_j$, for $i\neq j$, for otherwise we would obtain $x_i=x_j\in\Psi$---contradictory with our assumptions on $\Psi$. 
    
    We are now going to show that $\Phi$ necessarily contains a formula $x_i=y_i$, for some $i\in\{1,\dots,n\}$. Suppose that $\Phi$ does not contain any such formula. Since $\sim$ has more than one equivalence class, and since $\bU$ has quantifier elimination for primitive positive formulae, $\Phi$ contains an atomic  formula $\varphi(\tx,\ty)=\varrho(z_1,\dots,z_k)$, where $z_1,\dots,z_k\in\{x_1,\dots,x_n,y_1,\dots,y_n\}$, and where $\{z_1,\dots,z_k\}\cap\{x_1,\dots,x_k\}$ and $\{z_1,\dots,z_k\}\cap\{y_1,\dots,y_k\}$ are both nonempty. By Lemma~\ref{oneGeneratedRel}, there exists $\tu\tv\in\sigma^\sim$, such that $\Tpp_\bU(\tu\tv)=\Phi$. Moreover, we have $\Tpp_\bU(\tu)=\Tpp_\bU(\tv)=\Psi$. Let $\bU$ and $\bW$ the substructures of $\bU$ induced by $U=\{u_1,\dots,u_n\}$ and $W=U\cup \{v_1,\dots,v_n\}$, respectively. Let $\bW'$ be an isomorphic copy of $\bW$ such that $W'=U\cup\{v_1',\dots,v_n'\}$ and such that $W\cap W' =U$ and  are disjoint and such that $\iota\colon \bW\to\bW'$ defined through $\iota'\colon u_i\mapsto u_i,\,v_i\mapsto v_i'$ is an isomorphism. Then, since $\Age(\bU)$ has the free amalgamation property, we have that $\bW\oplus_\bU\bW'\in\Age(\bU)$. Thus, we can assume that $\bW\oplus_\bU\bW'\le\bU$. Let $\tv':=(v_1',\dots,v_n')$. Then by construction we have that $\Tpp_\bU^{(0)}(\tu\tv)=\Tpp_\bU^{(0)}(\tu\tv')$. Since $\bU$ has quantifier elimination for primitive positive formulae,  we also have $\Tpp_\bU(\tu\tv)=\Tpp_\bU(\tu\tv')$. Hence, $\tu\sim\tv'$. Since $\sim$ is symmetric and transitive, we have $\tv\sim\tv'$. Thus, we have $\varphi(\tx,\ty)\in\Tpp_\bU(\tv\tv')$. However, by the nature of the amalgamated free sum in free amalgamation classes, we have that $\varrho^\bW\cap \{v_1,\dots,v_n,v_1'\dots,v_n'\}^k=\emptyset$. With $\varrho^\bW=\varrho^\bU\cap W^k$, we arrive at a contradiction. Thus, our assumption was wrong and $\Phi$ contains a formula $x_{i_0}=y_{i_0}$ for some $i_0\in\{1,\dots,n\}$.
    
    Next we show that $\xi'$ is injective.  Without loss of generality we may assume that $\algebra{B}$ is equal to $\algebra{S}/_\sim$. Let $f,g\in\Pol^{(m)}(\bU)$ be two distinct functions. Then there exists $\ta=(a_1,\dots,a_m)\in U^m$, such that 
    \[ b:=f(a_1,\dots,a_m)\neq g(a_1,\dots,a_m)=:b'\]
    Since $\Aut(\bU)$ acts transitively on $U$, there exist 
    $\tc_1=(c_{1,1},\dots,c_{n,1}),\dots,\tc_m=(c_{1,m},\dots,c_{n,m})\in S$, 
    such that $c_{i_0,j}=a_j$, for each $j\in\{1,\dots,m\}$. Let 
    \[
    \tb=\begin{pmatrix}
    	b_1\\
    	\vdots\\
    	b_n
    \end{pmatrix}:=\begin{pmatrix}
    	f(c_{1,1},\dots,c_{1,m})\\
    	\vdots\\
    	f(c_{n,1},\dots,c_{n,m})
    \end{pmatrix}, \text{ and}\quad
	\tb'=\begin{pmatrix}
		b'_1\\
		\vdots\\
		b'_n
	\end{pmatrix}:=\begin{pmatrix}
    	g(c_{1,1},\dots,c_{1,m})\\
    	\vdots\\
    	g(c_{n,1},\dots,c_{n,m})
    \end{pmatrix}.
	\]
    Then $b_{i_0}=b\neq b'=b'_{i_0}$. Hence 
    $\xi'(f)([c_1]_\sim,\dots,[c_n]_\sim) = [\tb]_\sim\neq[\tb']_\sim = \xi'(g)([c_1]_\sim,\dots,[c_n]_\sim)$.
    Thus, $\xi'$ is injective (and hence bijective). 
    
    It remains to show that $\xi'$ is open. Let $a_0,\dots,a_k\in U$, and let $N$ be the basic clopen subset of $\Pol(\bU)$ that consists of all functions $f\in\Pol^{(k)}(\bU)$ with the property that $f(a_1,\dots,a_k)=a_0$. Let us define 
    \[ A:=\{[(b_1,\dots,b_n)]_\sim\in S/_\sim\mid b_{i_0}=a_0\}.\]
    Because $\Aut(\bU)$ acts transitively on $U$, it follows that $A$ is non-empty. For every $j\in\{1,\dots,k\}$ let $\tc_j=(c_{j,1},\dots,c_{j,n})$ be an element of $S$, such that $c_{j,i_0}=a_j$. Again, the existence of these tuples follows from the transitivity of $\Aut(\bU)$. We are going to show now that for all $f\in\Pol^{(k)}(\bU)$ we have 
    \[ 
    f(a_1,\dots,a_k)=a_0 \iff \xi'(f)([\tc_1]_\sim,\dots,[\tc_k]_\sim)\in A.
    \] 
    Indeed, if $f(a_1,\dots,a_k)=a_0$, then
    \[
    \xi'(f)([\tc_1]_\sim,\dots,[\tc_k]_\sim)=\begin{bmatrix}
    	f(c_{1,1},\dots,c_{1,k})\\
    	\vdots\\
    	f(c_{i_0,1},\dots,c_{i_0,k})\\
    	\vdots\\
    	f(c_{n,1},\dots,c_{n,k})
    \end{bmatrix}_\sim = 
\begin{bmatrix}
    	f(c_{1,1},\dots,c_{1,k})\\
    	\vdots\\
    	f(a_1,\dots,a_k)\\
    	\vdots\\
    	f(c_{n,1},\dots,c_{n,k})
    \end{bmatrix}_\sim    \]
    Thus, $\xi'(f)([\tc_1]_\sim,\dots,[\tc_k]_\sim)\in A$. 
    
    If, on the other hand, $\xi'(f)([\tc_1]_\sim,\dots,[\tc_k]_\sim)\in A$, then 
    $f(a_1,\dots,a_k)=f(c_{i_0,1},\dots,c_{i_0,k})= a_0$.
    Thus, we obtain that 
    \[
    	\xi'(N) = \bigcup_{[\tc_0]_\sim\in A} \{\xi'(f)\mid f\in\Pol^{(k)}(\bU), \xi'(f)([c_1]_\sim,\dots,[c_k]_\sim) = [c_0]_\sim\}. 
    \]
    Hence $\xi'(N)$ is open. This finishes the proof that $\xi'$ is open. 
\end{proof}
In order to make Proposition~\ref{autoopencriterion} applicable, we need a convenient criterion for a relational structure to have quantifier elimination for primitive positive formulae:
\begin{proposition}\label{qeppfcriterion}
	Let $\bU$ be a countable homogeneous $\omega$-categorical relational structure such that
	\begin{enumerate}
		\item $\Age(\bU)$ has the free amalgamation property,
		\item $\Age(\bU)$ is closed with respect to finite products,
		\item $\Age(\bU)$ has the \HAP.
	\end{enumerate}
	Then $\bU$ has quantifier elimination for primitive positive formulae.
\end{proposition}
Before proving this proposition, we need to recall a result about retracts of homogeneous structures:
\begin{proposition}\label{uhretract}
	Let $\bU$ be a countable homogeneous relational structure, and let $\bT\in\overline{\Age(\bU)}$, such that 
	\begin{enumerate}
		\item\label{cond1} for all $\bA,\bB_1,\bB_2\in\Age(\bU)$, $f_1\colon \bA\injto\bB_1$, $f_2\colon \bA\injto\bB_2$, $h_1\colon \bB_1\to\bT$, $h_2\colon \bB_2\to\bT$, if $h_1\circ f_1=h_2\circ f_2$, then there exists $\bC\in\Age(\bU)$, $g_1\colon \bB_1\injto\bC$, $g_2\colon \bB_2\injto\bC$, $h\colon \bC\to\bT$, such that the following diagram commutes:
		\[
		\begin{tikzcd}
			& & \bT \\
			\bB_1 \arrow[bend left]{urr}{h_1}\rar[hook,dashed]{g_1}& \bC\urar[dashed]{h}\\
			\bA \uar[hook]{f_1}\rar[hook]{f_2}& \bB_2.\arrow[bend right]{uur}[swap]{h_2}\uar[hook,dashed]{g_2}
		\end{tikzcd}
		\]
		\item\label{cond2}  
		for all $\bA,\bB\in\Age(\bU)$, $\iota\colon \bA\injto\bB$, $h\colon \bA\to\bT$ there exists $\hat{h}\colon \bB\to\bT$ such that the following diagram commutes:
		\[
		\begin{tikzcd}
			\bA \rar{h} \dar[hook,swap]{\iota}& \bT \\
			\bB \urar[swap,dashed]{\hat{h}}
		\end{tikzcd}
		\]
	\end{enumerate}
	Then $\bT$ is isomorphic to a retract of $\bU$.
\end{proposition}
\begin{proof}
	This follows directly from \cite[Theorem 4.2]{PecPec13a}.
\end{proof}

\begin{proof}[Proof of Proposition~\ref{qeppfcriterion}]
	We are going to show that $\bU$ is polymorphism homogeneous (in the sense of \cite{PecPec15}). Then it follows from  \cite[Corollary 3.13]{PecPec15} and the assumption that $\bU$ is $\omega$-categorical, that $\bU$ has quantifier elimination for primitive positive formulae. 
	
	In order to show that $\bU$ is polymorphism homogeneous, we are going to show that all finite powers of $\bU$ are homomorphism homogeneous. After that it follows from \cite[Proposition 2.1]{PecPec15}, that $\bU$ is polymorphism homogeneous.
	
	In order to show that every finite power of $\bU$ is homomorphism homogeneous, we are first going to argue that $\bU$ is homomorphism homogeneous (this follows from \cite[Proposition 3.8]{Dol14}; note that the 1PHEP mentioned in this paper is equivalent to the \HAP). Then we will show that every finite power of $\bU$ is in fact isomorphic to a retract of $\bU$. Finally, it follows from the folklore fact that retracts of homomorphism homogeneous structures are homomorphism homogeneous, that all finite powers of $\bU$ are homomorphism homogeneous.
	
	In order to show that every finite power of $\bU$ is isomorphic to a retract of $\bU$, we will make use of Proposition~\ref{uhretract}. First of all, since $\Age(\bU)$ has the free amalgamation property, condition \ref{cond1} of Proposition~\ref{uhretract} is satisfied for every structure $\bT$, younger than $\bU$. We simply need to choose $\bC$ to be equal to $\bB_1\oplus_\bA\bB_2$. 
	
	Let us verify condition \ref{cond2} of Proposition~\ref{uhretract}, when $\bT=\bU^n$: Let $\bA,\bB\in\Age(\bU)$, let $\iota\colon \bA\injto\bB$ be an embedding, and let $h\colon \bA\to\bU^n$ be a homomorphism. For every $i\in\{1,\dots,n\}$ let $h_i\colon \bA\to\bU$ be defined through $h_i:=e_i^n\circ h$. Since $\bU$ is homomorphism homogeneous, it follows that it is also weakly homomorphism homogeneous. Thus, for every $i\in\{1,\dots,n\}$, there exists a homomorphism $\hat{h}_i\colon \bB\to\bU$, such that $\hat{h}_i\circ\iota=h_i$. Now we may define $\hat{h}$ according to
	\[
	\hat{h}:=\langle h_1,\dots,h_n\rangle \colon  \bB\to\bU^n \colon  b\mapsto (\hat{h}_1(b),\dots,\hat{h}_n(b)). 
	\]
	Clearly, with this definition we have $\hat{h}\circ\iota= h$. Thus, we may apply Proposition~\ref{uhretract} to the case $\bT=\bU^n$, and we obtain that $\bU^n$ is isomorphic to a retract of $\bU$. 
\end{proof}

\begin{proof}[Proof of Proposition~\ref{autohomhap}]
	This immediately follows from Proposition~\ref{autoopencriterion}, together with Proposition~\ref{qeppfcriterion}.
\end{proof}

\begin{remark}
	Retracts of homogeneous structures were considered also by Dolinka and Kubi\'s (\cite{Dol12,Kub15}). 
\end{remark}

\subsection{Existence of strong gate coverings}
The hardest part in both our strategies for showing automatic homeomorphicity is to prove the existence of a strong gate covering. A major part of the  rest of the paper will be devoted to this task.
\begin{definition}
  Let $\bU$ be a structure.  An $n$-ary  polymorphism $u$ of $\bU$ is called \emph{universal} if for all structures $\bA\in\overline{\Age(\bU)}$ and for every homomorphism $f\colon  \bA^n\to\bU$ there exist $\iota\colon  \bA\injto\bU$ such that for all $(a_1,\dots,a_n)\in\bA^n$ holds $f(a_1,\dots,a_n)= u(\iota(a_1),\dots,\iota(a_n))$.
\end{definition}

\begin{definition}
  Let $\bU$ be a structure.  An $n$-ary polymorphism $u$ of $\bU$ is called \emph{homogeneous} if for all structures $\bA\in\Age(\bU)$, for every    homomorphism $f\colon  \bA^n\to\bU$, for all embeddings $\iota_1,\iota_2\colon  \bA\injto \bU$ with
    \begin{equation*}
      \forall (a_1,\dots, a_n)\in A^n:\,
      u(\iota_1(a_1),\dots,\iota_1(a_n)) = f(a_1,\dots,a_n) = u(\iota_2(a_1),\dots,\iota_2(a_n))
    \end{equation*}
    there exists $h\in\Aut(\bU)$ such that
    \begin{enumerate}
    \item $h\circ \iota_{1}=\iota_{2}$,
    \item for all $(a_1,\dots,a_n)\in U^n$ we have $u(h(a_1),\dots,h(a_n)) = u(a_1,\dots,a_n)$.
    \end{enumerate}
\end{definition}

\begin{lemma}\label{closeiota}
  Let $\bU$ be a relational structure that has an $n$-ary universal homogeneous  polymorphism $u$. Let $A\subseteq U$ be finite. Let $f,g$ be $n$-ary  polymorphisms of $\bU$ that agree on $A^n$. Then there exist selfembeddings $\iota_1$ and $\iota_2$, such that
  \begin{enumerate}
  \item $f(x_1,\dots,x_n)=u(\iota_1(x_1),\dots,\iota_1(x_n))$,
  \item $g(x_1,\dots,x_n)=u(\iota_2(x_1),\dots,\iota_2(x_n))$,
  \item $\iota_1\restr_A=\iota_2\restr_A$.
  \end{enumerate}
\end{lemma}
\begin{proof}
  Since $u$ is universal, there exist $\iota_1,\iota_2\colon  \bU\injto\bU$, such that for all $(x_1,\dots,x_n)\in U^n$ we have 
  \[
    f(x_1,\dots,x_n)=u(\iota_1(x_1),\dots,\iota_1(x_n)), \text{ and}\quad
    g(x_1,\dots,x_n)=u(\iota_2(x_1),\dots,\iota_2(x_n))
  \]
  Let $\hat\iota_i:=\iota_i\restr_A$, for $i\in\{1,2\}$, and let $\hat{f}:=f\restr_{A^n}$. Let $(a_1,\dots,a_n)\in A^n$. Then we compute
  \begin{equation*}
    \hat{f}(a_1,\dots,a_n) = f(a_1,\dots,a_n)=u(\iota_1(a_1),\dots,\iota_1(a_n))=
    u(\hat\iota_1(a_1),\dots,\hat\iota_1(a_n)).
  \end{equation*}
  Moreover,
  \begin{equation*}
    \hat{f}(a_1,\dots,a_n) = f(a_1,\dots,a_n)=g(a_1,\dots,a_n) =u(\iota_2(a_1),\dots,\iota_2(a_n))=
    u(\hat\iota_2(a_1),\dots,\hat\iota_2(a_n)).
  \end{equation*}
  Since $u$ is homogeneous, there exists an  automorphism $h$ of $\bU$, such that $h\circ\hat\iota_1=\hat\iota_2$, and such that for all $(a_1,\dots,a_n)\in U^n$
  we have $u(h(a_1),\dots, h(a_n))=u(a_1,\dots,a_n)$. 
  Let $\tilde\iota_1:=h\circ\iota_1$. Then $\tilde\iota_1\restr_A=h\circ\hat\iota_1=\hat\iota_2=\iota_2\restr_A$. Moreover, for all $(a_1,\dots,a_n)\in U^n$, we have
  \begin{equation*}
    u(\tilde\iota_1(a_1),\dots,\tilde\iota_1(a_n))=u(h(\iota_1(a_1)),\dots, h(\iota_1(a_n)))=  u(\iota_1(a_1),\dots, \iota_1(a_n))= f(a_1,\dots,a_n).\qedhere
  \end{equation*}
\end{proof}

\begin{proposition}\label{contiota}
  Let $\bU$ be a countably infinite relational structure that has an $n$-ary universal homogeneous polymorphism $u$. Let $(f_j)_{j<\omega}$ be a sequence of $n$-ary polymorphisms of $\bU$ that converge to an $n$-ary polymorphism $f$ of $\bU$. Then there is a sequence $(\iota_j)_{j<\omega}$ of selfembeddings of $\bU$, and a selfembedding $\iota$ of $\bU$, such that
  \begin{enumerate}
  \item for every $j<\omega$ and for all $(x_1,\dots,x_n)\in U^n$ we
    have
    $f_j(x_1 ,\dots,x_n) = u(\iota_j(x_1),\dots,\iota_j(x_n))$,
  \item $(\iota_j)_{j<\omega}$ converges to $\iota$,
  \item for all $(x_1,\dots, x_n)\in U^n$ we have  $f(x_1 ,\dots,x_n) = u(\iota(x_1),\dots,\iota(x_n))$.
  \end{enumerate}
\end{proposition}
\begin{proof}
  Since $u$ is universal, there exists a selfembedding $\iota$ of $\bU$ such that for every $(x_1,\dots,x_n)\in U^n$ we have  $f(x_1,\dots,x_n)=u(\iota(x_1),\dots,\iota(x_n))$.

  Suppose that the ultrametric $d_U$ on $\clone{O}_U^{(1)}$ is induced by the enumeration $(u_i)_{i<\omega}$ of $U$, and that $d_U$ on $\clone{O}_U^{(n)}$  is induced by the enumeration $(\tv_i)_{i<\omega}$  of $U^n$. For every finite subset $A$ of $U$ let $m_A$  be the smallest element of $\omega$ such that $A^n\subseteq \{\tv_0,\dots,\tv_{m_A-1}\}$. For every $i<\omega$, let $A_i:=\{u_0,\dots,u_{i-1}\}$.    Then $\bigcup_i A_i^n=U^n$ and thus the sequence $(m_{A_i})_{i<\omega}$ is monotonous and unbounded.  

  Since $(f_j)_{j<\omega}$ converges to $f$, for every $i<\omega$  there exists a $j_i<\omega$ such that for every $k>j_i$ we have that $D_U^{(n)}(f_k,f)>m_{\bA_i}$.  Without loss of generality we may assume that $j_i$ is chosen as small as possible.

  For $0\le k< j_0$, using the fact that $u$ is universal, we choose  $\iota_k$, such that for all $(x_1,\dots,x_n)\in U^n$
  \[f_k(x_1,\dots,x_n)= u(\iota_k(x_1),\dots\iota_k(x_n)).\]

  For $j_i\le k < j_{i+1}$, using Lemma~\ref{closeiota}, we chose $\iota_k$, such that for all $(x_1,\dots,x_n)\in U^n$ 
  \[f_k(x_1,\dots,x_n)= u(\iota_k(x_1),\dots\iota_k(x_n)).\]
  and such that $\iota_k$ agrees with $\iota$ on $A_i$.

  It remains to observe that, the sequence $(\iota_j)_{j<\omega}$ converges to $\iota$. Let $\varepsilon>0$ and let
  \[
  N:=\max(-\lfloor\log_2(\varepsilon)\rfloor,1).
  \]   
  Then, by construction, for all $k\ge j_N$, we have that $\iota_k$ agrees with $\iota$ on $\{u_0,\dots,u_{N-1}\}$---in particular,  $D_U^{(n)}(\iota_k,\iota)\ge N$, and thus $d_U(\iota_k,\iota)\le\varepsilon$.
\end{proof}

\begin{proposition}\label{stronggateexistence}
    If $\bU$ is a relational structure that has a  $k$-ary universal homogeneous polymorphism $u_k$ for every $k\in\bN\setminus\{0\}$, then $\Pol(\bU)$ has a strong gate covering.
\end{proposition}
\begin{proof}
    This is a direct consequence of Proposition~\ref{contiota}, taking the set  $\cover{U}=\{\Pol^{(k)}(\bU)\mid k\in\bN\setminus\{0\}\}$ as an open covering of $\Pol(\bU)$, and for $U=\Pol^{(k)}(\bU)$ putting $f_U:=u_k$. 
\end{proof}

\section{Existence of universal homogeneous polymorphisms}
Above, we saw, how the existence of universal homogeneous polymorphisms leads to the existence of strong gate coverings. In this section we derive necessary and sufficient conditions for a relational structure to have universal homogeneous polymorphisms. 
In order to achieve this goal, we will make use of axiomatic \Fraisse theory as it was introduced by Droste and G\"obel in \cite{DroGoe92}. As this theory is not yet in the folklore, we will recall its most important features.

\subsection{Universal homogeneous objects in categories}
\begin{definition}
    Let $\cat{C}$ be a category in which all morphisms are monomorphisms, and let $\cat{C}^*$ be a full subcategory of $\cat{C}$. An object $U$ of $\cat{C}$ is called 
\begin{description}
\item[$\cat{C}$-universal] if for every $A\in\cat{C}$ there is a morphism $f\colon  A\to U$,
\item[$\cat{C}^*$-homogeneous] if for every $A\in\cat{C}^*$ and for all $f,g\colon  A\to U$ there exists an automorphism $h$ of $U$ such that $h\circ f=g$,
\item[$\cat{C}^*$-saturated] if for every $A,B\in\cat{C}^*$ and for all $f\colon  A\to U$, $g\colon  A\to B$ there exists some $h\colon  B\to U$ such that $h\circ g=f$.
\end{description}
\end{definition}

\begin{example}
    Let $\bU$ be a countably infinite relational structure. Consider the category $\cat{C}$ with objects
    \[ \{ f\colon  \bA^n\to\bU\mid \bA\in\overline{\Age(\bU)}\}.\]
    For objects $f\colon  \bA^n\to\bU$ and $g\colon  \bB^n\to \bU$ the morphisms in $\cat{C}$ from $f$ to $g$ are embeddings $\iota\colon  \bA\injto\bB$, with the property that the following diagram commutes: 
      \[
  \begin{tikzcd}
    \bB^n  \rar{g} &\bU \\
    \bA^n. \uar[hook]{\iota^n}\urar{f} 
  \end{tikzcd}
  \]
    In other words, for every $(a_1,\dots,a_n)\in A^n$ we have     $f(a_1,\dots,a_n)= g(\iota(a_1),\dots,\iota(a_n))$. Let $\cat{C}^*$ be the full subcategory of $\cat{C}$ that is spanned by  $\{ f\colon  \bA^n\to\bU\mid \bA\in\Age(\bU)\}$. Now we have that a homomorphism $h\colon  \bU^n\to\bU$ of $\bU$ is $\cat{C}$-universal if and only if $h$ is an $n$-ary universal polymorphism of $\bU$. Moreover, $h$ is an $n$-ary homogeneous polymorphism of $\bU$ if and only if $h$ is $\cat{C}^*$-homogeneous. 
    
    Be aware that $\cat{C}$ may contain a $\cat{C}$-universal, $\cat{C}^*$-homogeneous object $u\colon  \bV^n\to\bU$, but that $\bV$ is not isomorphic to $\bU$. In the sequel it is going to be our task to give conditions on $\cat{C}$ to have universal homogeneous objects and to give conditions, when there is one such object whose domain is equal to $\bU^n$.
\end{example}

\subsection{The Droste-G\"obel Theorem}
\begin{definition}
  Let $\cat{C}$ be a category and let $\lambda$ be an ordinal number. Then $(\lambda, \leq)$ can be considered as a category in the usual way. The functors from $(\lambda, \leq)$ to $\cat{C}$ are called \emph{$\lambda$-chains} of $\cat{C}$.
\end{definition}
\begin{definition}
  Let $\cat{C}$ be a category and let $\lambda$ be a regular cardinal number. An object $A$ of $\cat{C}$ is called \emph{$\lambda$-small} if for every $\lambda$-chain $F\colon  (\lambda, \leq)\to \cat{C}$ with limiting cocone $(S, {(f_i)}_{i<\lambda})$ and for every morphism $h\colon  A\to S$ there exists a $j<\lambda$ and a $g\colon  A\to F(j)$, such that $h=f_j\circ g$.
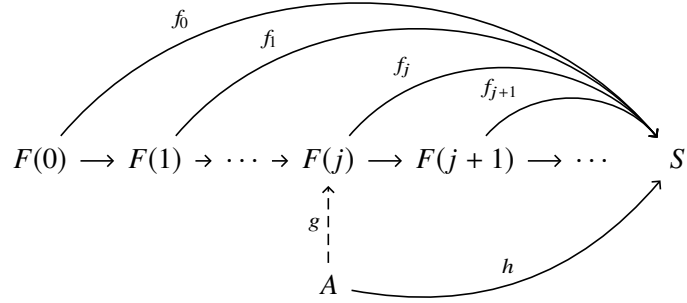
  \[
  \begin{tikzcd}[column sep=tiny,row sep=large]
    F(0)\arrow{rr}\arrow[bend left=50]{rrrrrrrrrr}[near start]{f_0} &
    & F(1)\rar \arrow[bend left=50]{rrrrrrrr}[near start]{f_1} &\cdots
    \rar& F(j)\arrow[bend left=50]{rrrrrr}[near start]{f_j}
    \arrow{rr}& & F(j+1)\arrow{rr}{}
    \arrow[bend left=50]{rrrr}[near start]{f_{j+1}} & &\cdots & & S\\
    & & & & A\uar[dashed]{g}\arrow[bend right]{rrrrrru}{h}
  \end{tikzcd}
  \]
  The full subcategory of $\cat{C}$, spanned by all $\lambda$-small objects, will be denoted by $\cat{C}_{<\lambda}$.
\end{definition}

\begin{definition}
  A category $\cat{C}$ is called \emph{semi-$\lambda$-algebroidal}, if:
  \begin{enumerate}
  \item all $\mu$-chains ($\mu\le\lambda$) in $\cat{C}_{<\lambda}$ have a    colimit in $\cat{C}$.
  \item every object in $\cat{C}$ is the colimit of a $\lambda$-chain in $\cat{C}_{<\lambda}$.
  \end{enumerate}
  It is called \emph{$\lambda$-algebroidal}, if in addition $\cat{C}_{<\lambda}$ has up to isomorphism at most $\lambda$ objects and between any two objects of $\cat{C}_{<\lambda}$ there are at most $\lambda$ morphisms.
\end{definition}

\begin{example}\label{exmpllambdaalg} Let $\lambda$ be a regular cardinal.
  \begin{enumerate}
  \item The category of sets of cardinality $\le \lambda$ with injective functions is $\lambda$-algebroidal. The $\lambda$-small sets are the sets of cardinality less than $\lambda$.
  \item If $\bA$ is a countably infinite structure, then $(\overline{\Age(\bA)},\injto)$ is an $\omega$-algebroidal category.  The $\omega$-small objects in this category are the elements of $\Age(\bA)$.
  \item Groups (considered as categories with just one object) are $\lambda$-algebroidal.
  \end{enumerate}
\end{example}

\begin{definition}
    Let $\cat{C}$ be a category in which all morphisms are monomorphisms, and let $\cat{C}^*$ be a full subcategory of $\cat{C}$. We say that
\begin{description}
\item[$\cat{C}^*$ has the joint embedding property] if for all $A,B\in\cat{C}^*$ there exists a $C\in\cat{C}^*$ and morphisms $f\colon  A\to C$ and $g\colon  B\to C$,
\item[$\cat{C}^*$ has the amalgamation property] if for all $A$, $B$, $C$ from $\cat{C}^*$ and $f\colon  A\to B$, $g\colon  A\to C$, there exists $D\in\cat{C}^*$ and $\hat{f}\colon  C\to D$, $\hat{g}\colon  B\to D$ such that the following  diagram commutes:
  \[
  \begin{tikzcd}
    C \rar[dashed]{\hat{f}}& D\\
    A  \uar{g}\rar{f}&  B.\uar[dashed]{\hat{g}}
  \end{tikzcd}
  \]
\end{description}
\end{definition}

\begin{lemma}\label{JEPAP}
  Let $\cat{C}$ be a category that has the amalgamation property and that contains a weakly initial object. Then $\cat{C}$ has also the joint embedding property.
\end{lemma}
\begin{proof}
	This is clear.
\end{proof}

\begin{theorem}[Droste/G\"obel {\cite[Theorem 1.1]{DroGoe92}}]\label{DroGoe}
  Let $\lambda$ be a regular cardinal, and let $\cat{C}$ be a $\lambda$-algebroidal category in which all morphisms are monomorphisms. Then, up to isomorphism, $\cat{C}$ contains at most one $\cat{C}$-universal, $\cat{C}_{<\lambda}$-homogeneous object. Moreover, $\cat{C}$ contains a $\cat{C}$-universal, $\cat{C}_{<\lambda}$-homogeneous object if and only if $\cat{C}_{<\lambda}$ has the joint embedding property and the amalgamation property.
\end{theorem}

\begin{proposition}[{\cite[Proposition 2.2]{DroGoe92}}]\label{univhomsat}
  Let $\lambda$ be a cardinal and let $\cat{C}$ be a semi-$\lambda$-algebroidal category in which all morphisms are monic. Then for any object $U$ of $\cat{C}$ the following are equivalent:
  \begin{enumerate}[label=(\arabic*)]
  \item $U$ is $\cat{C}$-universal and $\cat{C}_{<\lambda}$-homogeneous,
  \item $U$ is $\cat{C}_{<\lambda}$-universal and $\cat{C}_{<\lambda}$-homogeneous,
  \item $U$ is $\cat{C}_{<\lambda}$-universal and $\cat{C}_{<\lambda}$-saturated.
  \end{enumerate}
  Moreover, any two $\cat{C}$-universal, $\cat{C}_{<\lambda}$-homogeneous objects in $\cat{C}$ are isomorphic. Finally, if $\cat{C}_{<\lambda}$ contains a weakly initial object, then every $\cat{C}_{<\lambda}$-saturated object is $\cat{C}_{<\lambda}$-universal.
\end{proposition}

\subsection{Universal homogeneous objects in comma categories}

\begin{definition}
  Let $\cat{A}$,$\cat{B}$,$\cat{C}$ be categories, let $F\colon  \cat{A}\to\cat{C}$, $G\colon  \cat{B}\to\cat{C}$ be functors. The comma category $(F \komma G)$ has as objects triples $(A,f,B)$ where $A\in\cat{A}$, $B\in\cat{B}$, $f\colon  FA\to GB$. The morphisms from $(A,f,B)$ to $(A',f',B')$ are pairs $(a,b)$ such that $a\colon  A\to A'$ in $\cat{A}$, and  $b\colon  B\to B'$ in $\cat{B}$,  such that the following diagram commutes:
  \[
  \begin{tikzcd}
    FA'\rar{f'}& GB'\\
    FA \uar{Fa}\rar{f}& GB.\uar{Gb}
  \end{tikzcd}
  \]
\end{definition}

\begin{definition}
  Let $\cat{A}$, $\cat{B}$, $\cat{C}$ be categories, $F\colon  \cat{A}\to\cat{C}$, $G\colon  \cat{B}\to\cat{C}$ be functors. We say that $(F,G)$ has property
  \begin{enumerate}[label=\textbf{(P\arabic*)},ref=(P\arabic*)]
  \item\label{P1} if $\cat{A}$ and $\cat{B}$ are $\lambda$-algebroidal,
  \item\label{P2} if all morphisms of $\cat{A}$ and $\cat{B}$ are monomorphisms,
  \item\label{P3} if $F$ preserves colimits of $\lambda$-chains,
  \item\label{P4} if $\forall\mu<\lambda$: $F$ preserves colimits of $\mu$-chains of $\lambda$-small objects in $\cat{A}$,
  \item\label{P5} if $G$ preserves colimits of $\lambda$-chains of $\lambda$-small objects in $\cat{B}$,
  \item\label{P6} if $G$ preserves monomorphisms,
  \item\label{P7} if whenever $H$ is a $\lambda$-chain in $\cat{B}$ with limiting cocone $(B,(g_i)_{i<\lambda})$, and $A\in\cat{A}_{<\lambda}$, then for every $f\colon  FA\to GB$ there exists a $j<\lambda$ and an $h\colon  FA\to GH(j)$, such that $Gg_j\circ h=f$.
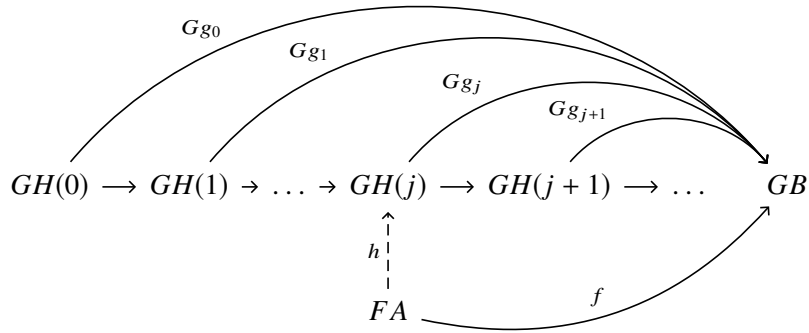
    \[
    \begin{tikzcd}[column sep=tiny,row sep=large]
      GH(0)\arrow{rr}\arrow[bend left=50]{rrrrrrrrrr}[near
      start]{Gg_0} & & GH(1)\rar \arrow[bend left=50]{rrrrrrrr}[near
      start]{Gg_1} &\dots \rar& GH(j)\arrow[bend left=50]{rrrrrr}[near
      start]{Gg_j} \arrow{rr}& & GH(j+1)\arrow{rr}{}
      \arrow[bend left=50]{rrrr}[near start]{Gg_{j+1}} & &\dots & & GB\\
      & & & & FA\uar[dashed]{h}\arrow[bend right]{rrrrrru}{f}
    \end{tikzcd}
    \]
  \item\label{P8} if for all $A\in \cat{A}_{<\lambda}$, $B\in\cat{B}_{<\lambda}$ there are at most $\lambda$ morphisms between $FA$ and $GB$ in $\cat{C}$.
  \end{enumerate}
\end{definition}

\begin{proposition}[{\cite[Propositions 2.15, 2.16]{PecPec13a}}]\label{lambdaalg}
  Let $\cat{A}$, $\cat{B}$, $\cat{C}$ be categories and let $F\colon  \cat{A}\to\cat{C}$, $G\colon  \cat{B}\to\cat{C}$ be functors. If $(F,G)$ has properties \ref{P1}--\ref{P7}, then $(F\komma G)$ is semi-$\lambda$-algebroidal. In this case, an object $(A,a,B)$ of $(F\komma G)$ is $\lambda$-small if and only if $A\in\cat{A}_{<\lambda}$ and $B\in \cat{B}_{<\lambda}$. If in addition $(F,G)$ has property \ref{P8}, then $(F\komma G)$ is $\lambda$-algebroidal.
\end{proposition}

\begin{lemma}\label{p567}
  Let $F\colon  \cat{A}\to\cat{C}$, $G\colon  \cat{B}\to \cat{C}$ be functors such that $\cat{B}$ consists just of one object and such that all morphisms of $\cat{B}$ are isomorphisms. Then $(F,G)$ has properties \ref{P5}, \ref{P6}, and \ref{P7}.
\end{lemma}
\begin{proof}
  \textbf{About~\ref{P6}:} In categories, every isomorphism is a monomorphism, and every functor preserves isomorphisms. Hence, since every morphism of $\cat{B}$ is an isomorphism, $G$ preserves monomorphisms.
  
  \textbf{About~\ref{P7}:}  
  Let $H\colon  (\lambda,\le)\to\cat{B}$ be a $\lambda$-chain with limiting cocone $(B,(g_i)_{i<\lambda})$ and let $A\in\cat{A}_{<\lambda}$. Moreover, let $f\colon  FA\to GB$. For an arbitrary $j<\lambda$ define $h=Gg_j^{-1}\circ f$. Then we have $Gg_j\circ h=f$. 
  
  \textbf{About~\ref{P5}:} Let $H\colon  (\lambda,\le)\to\cat{B}$ be a $\lambda$-chain with limiting cocone $(B,(g_i)_{i<\lambda})$ and let $(C, (c_i)_{i<\lambda})$ be a compatible cocone of $G\circ H$. Any mediating morphism $k\colon  GB\to C$ between $(GB,(Gg_i)_{i<\lambda})$ and $(C, (c_i)_{i<\lambda})$ has to fulfill the identities $k\circ Gg_j = c_j$ for all $j\in\lambda$. It follows that the only possibility to define $k$ is  $k:=c_0\circ Gg_0^{-1}$. With this choice we compute
  \begin{align*}
    k\circ Gg_j&=c_0\circ Gg_0^{-1}\circ Gg_j=c_0\circ(Gg_j\circ GH(0,j))^{-1}\circ Gg_j=c_0\circ GH(0,j)^{-1}\circ Gg_j^{-1}\circ Gg_j\\
    &=c_0\circ GH(0,j)^{-1}=c_j\circ GH(0,j)\circ GH(0,j)^{-1}=c_j.
  \end{align*}
  Thus, $(GB,(Gg_i)_{i<\lambda})$ is a limiting cocone of $G\circ H$. 
\end{proof}

\begin{definition}
  Let $\cat{A}$, $\cat{B}$, $\cat{C}$ be categories, $F\colon  \cat{A}\to\cat{C}$, $G\colon  \cat{B}\to\cat{C}$ be functors. We say that $(F,G)$ has property 
  \begin{enumerate}[label=\textbf{(P\arabic*)},ref=(P\arabic*),start=9]
  \item\label{P9} if for all $(B_1,h_1,T), (B_2,h_2,T)\in (F\komma G)_{<\lambda}$ there exists a $(C,h,T')\in(F\komma G)_{<\lambda}$ and morphisms $(f_1,g_1)\colon  (B_1,h_1,T)\to (C,h,T')$, $(f_2,g_2)\colon  (B_2,h_2,T)\to (C,h,T')$ such that the following diagram commutes:
    \[
    \begin{tikzcd}
      GT \rar[dashed]{Gg_1}& GT' & GT\lar[dashed][swap]{Gg_2}\\
      FB_1 \rar[dashed]{Ff_1}\uar{h_1}& FC\uar{h} & FB_2. \lar[dashed][swap]{Ff_2}\uar[swap]{h_2}
    \end{tikzcd}
    \]
  \item\label{P10} if for all $A,B_1,B_2\in\cat{A}_{<\lambda}$, $f_1\colon  A\to B_1$, $f_2\colon  A\to B_2$, $T\in\cat{B}_{<\lambda}$, $h_1\colon  FB_1\to GT$, $h_2\colon  FB_2\to GT$ with $h_1\circ Ff_1=h_2\circ Ff_2$ there exist $C\in\cat{A}_{<\lambda}$, $T'\in\cat{B}_{<\lambda}$, $g_1\colon  B_1\to C$, $g_2\colon  B_2\to C$, $h\colon  FC\to GT'$, $k\colon  T\to T'$ such that the following diagrams commute:
    \[
    \begin{tikzcd}
      &   & & & &GT\rar[dashed][swap]{Gk} & GT'\\
      B_1 \rar[dashed]{g_1}& C & & FB_1 \arrow[bend left]{urr}{h_1}
      \rar[dashed]{Fg_1}
      & FC\arrow[bend right,dashed]{urr}[swap,near end]{h}\\
      A \uar{f_1}\rar{f_2}& B_2\uar[dashed][swap]{g_2} & & FA
      \uar{Ff_1}\rar{Ff_2} & FB_2. \arrow[bend right]{uur}[swap]{h_2}
      \uar[dashed][swap]{Fg_2}
    \end{tikzcd}
    \]
  \item\label{P11} if for all $A,B\in\cat{A}_{<\lambda}$, $T_1\in\cat{B}_{<\lambda}$, $g\colon  A\to B$, $a\colon  FA\to GT_1$ there exist $T_2\in\cat{B}_{<\lambda}$, $h\colon  T_1\to T_2$, $b\colon  FB\to GT_2$ such that the following diagram commutes:
    \[
    \begin{tikzcd}
      FB \rar[dashed]{b}& GT_2\\
      FA\rar{a}\uar{Fg} & GT_1.\uar[swap,dashed]{Gh} 
    \end{tikzcd}
    \]
  \end{enumerate}
\end{definition}

\begin{proposition}[{\cite[Theorem 2.20]{PecPec13a}}]\label{homkomma}
  Let $F\colon  \cat{A}\to\cat{C}$, $G\colon  \cat{B}\to\cat{C}$ be functors. Suppose that $(F,G)$ fulfills conditions \ref{P1}--\ref{P8}. Then the following are true:
  \begin{enumerate}
  \item If $\cat{B}_{<\lambda}$ has the \JEP, then $(F\komma G)_{<\lambda}$ has the \JEP if and only if $(F,G)$ has property \ref{P9}.
  \item If $\cat{B}_{<\lambda}$ has the \AP, then $(F\komma G)_{<\lambda}$ has the \AP if and only if $(F,G)$ has property \ref{P10}.
  \end{enumerate}
\end{proposition}

\begin{proposition}[{\cite[Proposition 2.24]{PecPec13a}}]\label{firstobjsat}
  Let $F\colon  \cat{A}\to\cat{C}$, $G\colon  \cat{B}\to\cat{C}$ be functors such that $(F,G)$ fulfills conditions \ref{P1}--\ref{P8}. Additionally, suppose that $F$ is faithful, and that $(F\komma G)_{<\lambda}$ has the \JEP and the \AP. Let $(U,u,T)$ be an $(F\komma G)$-universal, $(F\komma G)_{<\lambda}$-homogeneous object in $(F\komma G)$.  Then $U$ is $\cat{A}_{<\lambda}$-saturated if and only if $(F,G)$ fulfills condition \ref{P11}.
\end{proposition}

\begin{proposition}\label{p10transfer}
  Let $F\colon  \cat{A}\to\cat{C}$, $G\colon  \cat{B}\to\cat{C}$ be functors such that $(F,G)$ fulfills conditions \ref{P1}--\ref{P7}. Suppose that $\cat{B}$ has a $\cat{B}_{<\lambda}$-universal object $V$. Let $\cat{V}$ be a $\lambda$-algebroidal subcategory of $\cat{B}$ that has $V$ as the only object and let $J\colon  \cat{V}\to\cat{B}$ be the identical embedding functor. Then $(F,G)$ fulfills condition \ref{P10} if $(F,G\circ J)$ does. Moreover, if $V$ is $\cat{B}_{<\lambda}$-saturated and $(F,G)$ fulfills condition \ref{P10}, then so does $(F,G\circ J)$.
\end{proposition}
\begin{proof}
  Suppose, $(F,G\circ J)$ fulfills condition \ref{P10}. Given $A,B_1,B_2\in\cat{A}_{<\lambda}$, $V'\in\cat{B}_{<\lambda}$, and morphisms $h_1,h_2,f_1,f_2$ such that $h_1\circ Ff_1 = h_2\circ Ff_2$.
  Since $V$ is $\cat{B}_{<\lambda}$-universal, there exists $\iota\colon  V'\to V$. Since $(F, G\circ J)$ fulfills condition \ref{P10}, there exist $C\in\cat{A}_{<\lambda}$ and morphisms $g_1,g_2,h,k$ such that the following diagram commutes:
  \begin{equation}\label{aep2}
  \begin{tikzcd}
    ~ & & GV\rar[dashed]{Gk} & GV\\
    FB_1 \arrow[bend left]{urr}{G\iota\circ h_1}\rar[dashed]{Fg_1}& FC\arrow[bend right,dashed,swap]{rru}[near end]{h}\\
    FA \uar{Ff_1}\rar[swap]{Ff_2} & FB_2\arrow[bend
    right,swap]{uur}{G\iota\circ h_2}\uar[dashed,swap]{Fg_2}
  \end{tikzcd}
  \end{equation}
  and such that $g_1\circ f_1=g_2\circ f_2$.

  Since $\cat{B}$ is $\lambda$-algebroidal, there exists a $\lambda$-chain $H\colon  (\lambda,\le)\to \cat{B}$ of $\lambda$-small objects in $\cat{B}$ and morphisms $\kappa_i\colon   Hi\to V$ ($i<\lambda$), such that $(V,(\kappa_i)_{i<\lambda})$ is a limiting cocone of $H$. Since $V'\in\cat{B}_{<\lambda}$, and $\iota\colon  V'\to V$, there exists $j_1<\lambda$ and $\tilde\iota\colon  V'\to Hj_1$ such that $\iota=\kappa_{j_1}\circ\tilde\iota$. Moreover, since $k\circ\iota\colon   V'\to V$, there exists $j_2<\lambda$ and $\tilde k\colon   V'\to Hj_2$ such that $k\circ\iota = \kappa_{j_2}\circ\tilde{k}$.
 
  Since $C\in\cat{A}_{<\lambda}$, $h\colon   FC\to GV$, and since $(F,G)$ fulfills condition \ref{P7}, there exists $j_3<\lambda$, $\tilde{h}\colon FC\to GHj_3$ such that $h= G\kappa_{j_3}\circ\tilde{h}$. Let $j$ be the maximum of $\{j_1,j_2,j_3\}$. Then we have
  \begin{align}
    \notag\iota &= \kappa_j\circ H(j_1,j)\circ\tilde{\iota}\\
    k\circ\iota &= \kappa_j\circ H(j_2,j)\circ\tilde{k}\label{kiota}\\
    h &= G\kappa_j\circ GH(j_3,j)\circ\tilde{h}\label{hcomp}
  \end{align}
  Let us define 
  \begin{align}
    \hat{k}&:=H(j_2,j)\circ\tilde{k},\label{hatk}\\ 
    \hat{h}&:=GH(j_3,j)\circ\tilde{h}\label{hath}  
  \end{align}
  It remains to show that the following diagram commutes:
  \[
  \begin{tikzcd}
    ~ & & GV'\rar{G\hat{k}} & GHj\\
    FB_1 \arrow[bend left]{urr}{h_1}\rar{Fg_1}& FC\arrow[bend right,swap]{urr}[near end]{\hat{h}} \\
    FA\uar{Ff_1}\rar[swap]{Ff_2} & FB_2.\uar[swap]{Fg_2}\arrow[bend
    right,swap]{uur}{h_2}
  \end{tikzcd}
  \]
  For this we calculate
  \begin{align*}
    G\kappa_j\circ G\hat{k}\circ h_1 &\stackrel{\eqref{hatk}}{=} G(\kappa_j\circ H(j_2,j)\circ\tilde{k})\circ h_1\stackrel{\eqref{kiota}}{=}Gk\circ G\iota\circ h_1\stackrel{\eqref{aep2}}{=} h\circ Fg_1 \\
    &\stackrel{\eqref{hcomp}}{=} G\kappa_j\circ GH(j_3,j)\circ\tilde{h}\circ Fg_1\stackrel{\eqref{hath}}{=} G\kappa_j\circ \hat{h}\circ Fg_1
  \end{align*}
  Since $\kappa_j$ is a monomorphism and since $G$ preserves monos, we conclude $G\hat{k}\circ h_1=\hat{h}\circ Fg_1$. Analogously one shows $G\hat{k}\circ h_2=\hat{h}\circ Fg_2$. Thus we showed that $(F,G)$ fulfills condition \ref{P10}.

  Suppose now that $V$ is $\cat{B}_{<\lambda}$-saturated and that $(F,G)$ fulfills condition \ref{P10}. Let $A,B_1,B_2\in\cat{A}_{<\lambda}$ and let $f_1,f_2,h_1,h_2$ be morphisms such that $h_1\circ Ff_1=h_2\circ Ff_2$.
  Since $\cat{B}$ is $\lambda$-algebroidal, there exists a $\lambda$-chain $H\colon  (\lambda,\le)\to\cat{B}$ of $\lambda$-small objects of $\cat{B}$ and morphisms $v_i\colon  Hi\to V$ ($i<\lambda$) such that $(V, (v_i)_{i<\lambda})$ is a limiting cocone of $H$. By condition \ref{P7}, there exist $j_1,j_2<\lambda$, $\tilde{h}_1\colon  FB_1\to GHj_1$, $\tilde{h}_2\colon  FB_2\to GHj_2$, such that $h_1=Gv_{j_1}\circ\tilde{h}_1$, $h_2=Gv_{j_2}\circ\tilde{h}_2$. Let $j$ be the maximum of $\{j_1,j_2\}$. Then 
  \begin{equation}
    h_1= Gv_j\circ GH(j_1,j)\circ\tilde{h}_1, \text{ and } 
     h_2=Gv_j\circ GH(j_2,j)\circ\tilde{h}_2.   \label{h1comp}   
  \end{equation}
 Let 
 \begin{equation}
    \hat{h}_1:=GH(j_1,j)\circ\tilde{h}_1, \text{ and let }
    \hat{h}_2:=GH(j_2,j)\circ\tilde{h}_2. \label{hath1}
 \end{equation}
 Since $(F,G)$ fulfills condition \ref{P10}, there exist $C\in\cat{A}_{<\lambda}$, $V'\in\cat{B}_{<\lambda}$, and morphisms $g_1,g_2,\hat{h},\hat{k}$ such that the following diagram commutes:
  \begin{equation}\label{aep3}
  \begin{tikzcd}
    ~& & GHj\rar[dashed]{G\hat{k}} & GV' \\
    FB_1\rar[dashed]{Fg_1}\arrow[bend left]{urr}{\hat{h}_1} & FC\arrow[bend right,dashed,swap]{urr}[near end]{\hat{h}}\\
    FA \uar{Ff_1}\rar[swap]{Ff_2}&
    FB_2.\uar[swap,dashed]{Fg_2}\arrow[bend right,swap]{uur}{\hat{h}_2}
  \end{tikzcd}
  \end{equation}
  Since $V$ is $\cat{B}_{<\lambda}$-saturated and since $v_j\colon  Hj\to V$ and $\hat{k}\colon   Hj\to V'$, there exists $\hat{v}_j\colon   V'\to V$ such that
  \begin{equation}
    v_j=\hat{v}_j\circ\hat{k}.\label{vj}
  \end{equation}
  It remains to show that the following diagram commutes:
  \[
  \begin{tikzcd}
    ~& & GV \\
    FB_1 \rar{Fg_1}\arrow[bend left]{rru}{h_1} & FC\urar[sloped, near end]{G\hat{v}_j\circ\hat{h}} \\
    FA \rar[swap]{Ff_2}\uar{Ff_1}& FB_2.\uar[swap]{Fg_2}\arrow[bend right,swap]{ruu}{h_2}
  \end{tikzcd}
  \]
  To this end we calculate:
  \begin{equation*}
    G\hat{v}_j\circ\hat{h}\circ Fg_1 \stackrel{\eqref{aep3}}{=} G\hat{v}_j\circ G\hat{k}\circ\hat{h}_1 \stackrel{\eqref{vj}}{=} Gv_j\circ \hat{h}_1\stackrel{\eqref{hath1}}{=} Gv_j\circ GH(j_1,j)\circ \tilde{h}_1 \stackrel{\eqref{h1comp}}{=} h_1.
  \end{equation*}
  Analogously one shows that $G\hat{v}_j\circ\hat{h}\circ Fg_2 = h_2$. Thus, $(F,G\circ J)$ fulfills condition \ref{P10}.
\end{proof}

\begin{proposition}\label{p11transfer}
  Let $F\colon  \cat{A}\to\cat{C}$, $G\colon  \cat{B}\to\cat{C}$ be functors such that $(F,G)$ fulfills conditions \ref{P1}--\ref{P7}. Suppose that $\cat{B}$ has a $\cat{B}_{<\lambda}$-universal object $V$. Let $\cat{V}$ be a subcategory of $\cat{B}$ that has $V$ as the only object and let $J\colon  \cat{V}\to\cat{B}$ be the identical embedding functor. Then $(F,G)$ fulfills condition \ref{P11} if $(F, G\circ J)$ does. Moreover, if $V$ is $\cat{B}_{<\lambda}$-saturated and if $(F, G)$ fulfills condition \ref{P11}, then so does $(F, G\circ J)$.
\end{proposition}
\begin{proof}
  Since $\cat{B}$ is $\lambda$-algebroidal, there exists a $\lambda$-chain $H\colon  (\lambda,\le)\to\cat{B}$ of $\lambda$-small objects in $\cat{B}$ and morphisms $v_i\colon  Hi\to V$ for every $i<\lambda$, such that $(V,(v_i)_{i<\lambda})$ is a limiting cocone for $H$.

  Suppose that $(F, G\circ J)$ fulfills condition \ref{P11}. Let $A,B\in\cat{A}_{<\lambda}$, $T\in\cat{B}_{<\lambda}$, $g\colon  A\to B$, $a\colon  FA\to GT$. Since $V$ is $\cat{B}_{<\lambda}$-universal, there exists $\iota\colon  T\to V$. Hence, by condition \ref{P11}, there exists $h\colon  V\to V$, $b\colon  FB\to GV$ such that the following diagram commutes:
  \begin{equation}\label{mHAPdiag}
  \begin{tikzcd}
    FB \arrow[dashed,swap]{rr}{b}& & GV\\
    FA \rar{a}\uar{Fg}& GT\rar{G\iota} & GV.\uar[swap,dashed]{Gh} 
  \end{tikzcd}
  \end{equation}
  By condition \ref{P7}, there exists $j_1<\lambda$, $\tilde{b}\colon    FB\to GHj_1$ such that $b=Gv_{j_1}\circ\tilde{b}$. Moreover, since $T\in\cat{B}_{<\lambda}$, there exists $j_2<\lambda$, $\tilde{h}\colon   T\to Hj_2$ such that $h\circ\iota= v_{j_2}\circ\tilde{h}$. Let $j$ be the maximum of $\{j_1,j_2\}$. Then 
  \begin{align}
    b&=Gv_j\circ GH(j_1,j)\circ\tilde{b}  \text{ and}\label{bcomp}\\
    h\circ\iota &= v_j\circ H(j_2,j)\circ\tilde{h}.\label{hiotacomp}  
  \end{align}
 Define 
 \begin{align}
    \hat{b}&:=GH(j_1,j)\circ\tilde{b} \text{ and}\label{bhatdef}\\ 
    \hat{h}&:=H(j_2,j)\circ\tilde{h}.\label{hathdef} 
 \end{align}
 It remains to observe that the following diagram commutes:
  \[
  \begin{tikzcd}
    FB \rar[dashed,swap]{\hat{b}}& GHj\\
    FA \rar{a}\uar{Fg}& GT.\uar[swap,dashed]{G\hat{h}}
  \end{tikzcd}
  \]
  Indeed, we compute
  \begin{align*}
    Gv_j\circ G\hat{h}\circ a &\stackrel{\eqref{hathdef}}{=} Gv_j\circ GH(j_2,j)\circ G\tilde{h}\circ a\stackrel{\eqref{hiotacomp}}{=} Gh\circ G\iota\circ a \stackrel{\eqref{mHAPdiag}}{=}b\circ Fg \\&\stackrel{\eqref{bcomp}}{=} Gv_j\circ GH(j_1,j)\circ \tilde{b}\circ Fg\stackrel{\eqref{bhatdef}}{=}Gv_j\circ\hat{b}\circ Fg.
  \end{align*}
  Since $v_j$ is a monomorphism and since $G$ preserves monos, we obtain $G\hat{h}\circ a= \hat{b}\circ Fg$. Thus $(F,G)$ fulfills condition \ref{P11}.
    
  Suppose now that $(F,G)$ fulfills condition \ref{P11} and that $V$ is $\cat{B}_{<\lambda}$-saturated. Let $A,B\in\cat{A}_{<\lambda}$, $g\colon   A\to B$, and $a\colon   FA\to GV$. Then, by condition \ref{P7}, there exists $j<\lambda$ and $\hat{b}\colon   FA\to GHj$ such that
  \begin{equation}\label{acomp}
    a=Gv_j\circ\hat{a}.  
  \end{equation}
 By condition \ref{P11}, there exists $V'\in\cat{B}_{<\lambda}$, $\hat{b}\colon   FB\to GV'$, $\hat{h}\colon   Hj\to V'$ such that the following diagram commutes:
  \begin{equation}\label{mHAP2diag}
  \begin{tikzcd}
    FB \rar[dashed,swap]{\hat{b}}& GV'\\
    FA \rar{\hat{a}}\uar{Fg}& GHj\rar{Gv_j}\uar[swap,dashed]{G\hat{h}} & GV.
  \end{tikzcd}
  \end{equation}
  Since $V$ is $\cat{B}_{<\lambda}$-saturated, there exists $\iota\colon    V' \to V$ such that 
  \begin{equation}\label{vjcomp}
    \iota\circ\hat{h}= v_j.  
  \end{equation}
 It remains to observe that the following diagram commutes:
  \[
  \begin{tikzcd}
    FB \drar{G\iota\circ\hat{b}}\\
    FA\uar{Fg}\rar{a} & GV.
  \end{tikzcd}
  \]
  Indeed, we compute
  \begin{equation*}
    G\iota\circ\hat{b}\circ Fg \stackrel{\eqref{mHAP2diag}}{=} G\iota\circ G\hat{h}\circ\hat{a}\stackrel{\eqref{vjcomp}}{=}Gv_j\circ\hat{a}\stackrel{\eqref{acomp}}{=}  a.
  \end{equation*}
  Thus, $(F, G\circ J)$ fulfills \ref{P11}.
\end{proof}

\subsection{Criteria for the existence of universal homogeneous polymorphisms}\label{sec3.4}
In the following we fix a signature $\uSigma$. With $\cat{C}_{\uSigma}$ we will denote the category of all $\uSigma$-structures with homomorphisms as morphisms. Moreover, we fix an arbitrary countably infinite $\uSigma$-structure $\bU$, and for every $n\in\bN\setminus\{0\}$ we denote by  $P_n\colon   (\overline{\Age(\bU)},\injto)\to\cat{C}_{\uSigma}$ the functor given by 
  $P_n\colon   \bA\mapsto\bA^n, f\mapsto f^n$. 
Finally, by $\cat{B}$ we will denote the category that has only one object $\bU$ and only one morphism $1_\bU$, and with $G$ we will denote the identical embedding functor from $\cat{B}$ to $\cat{C}_\uSigma$.

\begin{lemma}\label{prodcocont}
  With the notions from above, for every $n\in\bN\setminus\{0\}$, the functor $P_n$ preserves colimits of $\omega$-chains.
\end{lemma}
\begin{proof}
  We are going to make use of the fact that we know how colimits of chains may be constructed in $(\overline{\Age(\bU)},\injto)$ and in $\cat{C}_{\uSigma}$.

  Let $H\colon    (\omega,\le)\to(\overline{\Age(\bU)},\injto)$. Without loss of generality, we may assume that for all $j_1\leq j_2\in \omega$ we have that $Hj_1\le Hj_2$, and that $H(j_1,j_2)\colon    Hj_1\injto Hj_2$ is the identical embedding. For better readability, for every $j\in\omega$, we will denote $Hj$ by $\bV_j$.
 
  Let $\bV:=\bigcup_{j<\omega} \bV_j$ and let $v_{j}\colon \bV_j\injto \bV$ be the identical embedding. Then $(\bV,(v_{j})_{j\in\omega})$ is a limiting cocone of $H$.

  Note now that for all $j_1\le j_2<\omega$ we have that $P_n(H(j_1,j_2)) \colon    \bV_{j_1}^n\injto \bV_{j_2}^n$ is the identical embedding and that for every $j\in\omega$ we have that $P_n(v_j)\colon    \bV_j^n\injto \bV^n$ is the identical embedding, too. Moreover, $\bigcup_{j\in\omega} \bV_j^n= \bV^n$. Thus, $(\bV^n,(v_j^n)_{j\in\omega})$ is a limiting cocone of $P_n\circ H$. It follows that $P_n$ preserves colimits of $\omega$-chains.
\end{proof}

\begin{lemma}\label{p1p8OneObj}
  With the notions from above the comma-category  $(P_n\komma G)$ is $\omega$-algebroidal. 
\end{lemma}
\begin{proof}
  We already noted above (cf.\ Example~\ref{exmpllambdaalg}) that $(\overline{\Age(\bU)},\injto)$ and $\cat{B}$ are $\omega$-algebroidal. Moreover, by definition, all morphisms of $\cat{B}$ and $(\overline{\Age(\bU)},\injto)$ are monomorphisms. Thus, $(P_n, G)$ has properties \ref{P1} and \ref{P2}. By Lemma~\ref{prodcocont}, $(P_n, G)$ fulfills property \ref{P3}. Trivially, $P_n$ preserves colimits of finite chains. Thus $(P_n, G)$ satisfies property \ref{P4}. Now, by Lemma~\ref{p567}, $(P_n, G)$ fulfills properties \ref{P5}, \ref{P6}, \ref{P7}. 

  Let $\bA\in\Age(\bU)$. Then we have that  $P_n(\bA)=\bA^n$ is finite, too. Hence, since $\bU$ is countable, there are just countably many homomorphisms from $\bA^n$ to $\bU$. Thus, $(P_n, G)$ fulfills condition \ref{P8}.
  
  Now, by Proposition~\ref{lambdaalg}, $(P_n\komma G)$ is $\omega$-algebroidal. 
\end{proof}

\begin{lemma}\label{p1p8Age}
  With the notions from above, the comma-category $(P_n\komma P_1)$ is $\omega$-algebroidal. 
\end{lemma}
\begin{proof}
  We already noted above that $(\overline{\Age(\bU)},\injto)$ is $\omega$-algebroidal. Moreover, all morphisms of $(\overline{\Age(\bU)},\injto)$ are monomorphisms. Thus $(P_n, P_1)$ has properties \ref{P1} and \ref{P2}. By Lemma~\ref{prodcocont}, $(P_n, P_1)$ has properties \ref{P3} and \ref{P5}. Trivially, $P_n$ preserves colimits of finite chains. Thus $(P_n, P_1)$ fulfills property \ref{P4}.  Since every morphism of $(\overline{\Age(\bU)},\injto)$ is an embedding, every embedding is a monomorphism in $\cat{C}_\uSigma$, and since  $P_1$ is the identical embedding functor, we have that $(P_n, P_1)$ fulfills property \ref{P6}. 
  
  Since $P_n$ maps finite structures to finite structures, and since $P_1$ is the identical embedding functor, $(P_n, P_1)$ satisfies  property \ref{P7}. 
  
  Again, since $P_n$ maps finite structures to finite structures, $(P_n, P_1)$ has property \ref{P8}.
  
    Now, by Proposition~\ref{lambdaalg}, $(P_n\komma P_1)$ is $\omega$-algebroidal. 
\end{proof}

\begin{observation}\label{obscategories}
  With the notions from above, a polymorphism $u\colon   \bU^n\to\bU$  is universal and homogeneous if and only if $(\bU,u,\bU)$ is $(P_n\komma G)$-universal and $(P_n\komma G)_{<\omega}$-homogeneous.  
\end{observation}
    
\begin{definition}
  Let $\class{C}$ be a class of structures of the same type, and let $n\in\bN\setminus\{0\}$. We say that $\class{C}$ has the \emph{$\AEPn$} if for all $\bA,\bB_i, \bT\in\class{C}$, $f_i\colon   \bA\injto\bB_i$, $h_i\colon   \bB_i^n\to \bT$ (where $i\in\{1,2\}$), with $h_1\circ f_1^n= h_2\circ f_2^n$, there exist $\bC,\bT'\in\class{C}$, $g_i\colon   \bB_i\injto\bC$ (where $i\in\{1,2\}$),  $h\colon   \bC^n \to \bT'$, $k\colon   \bT\injto\bT'$ such that the following diagrams commute:
  \[
  \begin{tikzcd}
    &   & &  & \bT \rar[hook,dashed][swap]{k} & \bT'\\
    \bB_1 \rar[hook,dashed]{g_1}& \bC & \bB_1^n \arrow[bend
    left]{urr}{h_1} \rar[hook,dashed]{g_1^n}
    & \bC^n\arrow[bend right,dashed]{urr}[swap,near end]{h}\\
    \bA \uar[hook]{f_1}\rar[hook]{f_2}&
    \bB_2\uar[hook,dashed,swap]{g_2} & \bA^n
    \uar[hook]{f_1^n}\rar[hook]{f_2^n} & \bB_2^n. \arrow[bend
    right]{uur}[swap]{h_2} \uar[dashed,swap,hook]{g_2^n}
  \end{tikzcd}
  \]
\end{definition}

\begin{definition}
  Let $\class{C}$ be a class of structures of the same type, and let $n\in\bN\setminus\{0\}$. We say that $\class{C}$ has the \emph{\HAPn} if for all $\bA, \bB\in\class{C}$, $g\colon   \bA\injto\bB$, $\bT_1\in\class{C}$, $a\colon   \bA^n\to \bT_1$ there exist $\bT_2\in\class{C}$, $b\colon   \bB^n\to \bT_2$, $h\colon   \bT_1\injto\bT_2$ such that the following diagram commutes:
  \[
  \begin{tikzcd}
    \bB^n \rar[dashed]{b}& \bT_2\\
    \bA^n\rar{a}\uar[hook]{g^n} & \bT_1.\uar[swap,hook,dashed]{h} 
  \end{tikzcd}
  \]
	If $n=1$, then the \HAPn is just the \HAP.
\end{definition}
\begin{remark}
	Note that if $\class{C}$ is closed with respect to finite products, then it has the $\HAPn$ for every $n\in\bN\setminus\{0\}$ if and only if it has the $\HAP$.
\end{remark}

\begin{theorem}\label{univhompoly}
    Let $\bU$ be a countable homogeneous relational structure and let $n\in\bN\setminus\{0\}$. Then $\bU$ has an $n$-ary universal homogeneous polymorphism if and only if $\Age(\bU)$ has the \AEPn and the \HAPn.
\end{theorem}
\begin{proof}
  Consider the categories and functors from the beginning of Section~\ref{sec3.4}. From Lemmas~\ref{p1p8OneObj} and \ref{p1p8Age} it follows $(P_n,G)$, and $(P_n,P_1)$ are both $\omega$-algebroidal.
    
  ``$\Rightarrow$'': Suppose that $\Age(\bU)$ has the \AEPn and the \HAPn. Then we have that $(P_n,P_1)$ fulfills properties \ref{P10} and \ref{P11}.
    
  Note now that $\cat{B}$ is an $\omega$-algebroidal subcategory of $(\overline{\Age(\bU)},{\injto})$. Let $J\colon   \cat{B}\to(\overline{\Age(\bU)},\injto)$ be the identical embedding functor. Then $G=P_1\circ J$. By assumption, $\bU$ is both, $(\overline{\Age(\bU)},\injto)$-universal and $(\Age(\bU),\injto)$-homogeneous. Thus, from Proposition~\ref{univhomsat} it follows that $\bU$ is $(\Age(\bU),{\injto})$-saturated. Now we may conclude from Proposition~\ref{p10transfer}, that $(P_n, G)$ has property \ref{P10}.  Clearly, $\cat{B}_{<\omega}$ has the $\JEP$ and the $\AP$. Now, from Proposition~\ref{homkomma}, it follows that $(P_n\komma G)$ has the $\AP$. Note that $(\emptyset,\emptyset,\bU)$ is  an initial object in $(P_n\komma G_3)_{<\omega}$. Hence, by Lemma~\ref{JEPAP}, $(P_n\komma G)_{<\omega}$ has the \JEP. Now, from Proposition~\ref{lambdaalg} together with Theorem~\ref{DroGoe} it follows that there exists an $(P_n\komma G)$-universal, $(P_n\komma G)_{<\omega}$-homogeneous object $(\bV, w,\bU)$. From Proposition~\ref{p11transfer} it follows that $(P_n, G)$ has property \ref{P11}. Since $P_n$ is faithful, from Proposition~\ref{firstobjsat} we conclude that $\bV$ is $(\Age(\bU),\injto)$-saturated. Since $\emptyset$ is initial in $(\Age(\bU),\injto)$, and since all morphisms of $(\Age(\bU),\injto)$ are monomorphisms, from Proposition~\ref{univhomsat} it follows that $\bV$ is $(\overline{\Age(\bU)},\injto)$-universal and $(\Age(\bU),\injto)$-homogeneous. In other words, $\bV$ is universal and homogeneous with the same age like $\bU$. Thus, from \Fraisse's Theorem, it follows that there is an isomorphism $h\colon   \bU\to \bV$. Now define $u:= w\circ P_n(h)$. Then $(h, 1_\bU)\colon   (\bU, u, \bU)\to (\bV,w,\bU)$ is an isomorphism in $(P_n\komma G)$. In particular, $(\bU,u,\bU)$ is $(P_n\komma G)$-universal and $(P_n\komma G)_{<\omega}$-homogeneous. By Observation~\ref{obscategories}, $u$ is an $n$-ary universal homogeneous polymorphism of $\bU$.

  ``$\Leftarrow$'': Suppose that $\bU$ has an $n$-ary universal homogeneous polymorphism $u$. Then, by Observation~\ref{obscategories} $(\bU,u,\bU)$ is $(P_n\komma G)$-universal, $(P_n\komma G)_{<\omega}$-homogeneous. Since $\Age(\bU)$ has the \AP and the \JEP, it follows from Proposition~\ref{homkomma} that $(P_n,G)$ has properties \ref{P9} and \ref{P10}. Moreover, since $\bU$ is homogeneous, it follows from Proposition~\ref{univhomsat}, that it is $(\Age(\bU),\injto)$-saturated. Since $P_n$ is faithful, from Proposition~\ref{firstobjsat} it follows that $(P_n,G)$ has property \ref{P11}.

  $\bU$ is universal. In other words, it is $(\overline{\Age(\bU)},\injto)$-universal. Note also that $\cat{B}$ is a $\lambda$-algebroidal subcategory of $(\overline{\Age(\bU)},\injto)$. Now, from Propositions~\ref{p10transfer} and \ref{p11transfer} it follows that $(P_n,P_1)$ has properties \ref{P10}, \ref{P11}. However, this is the same as to say that $\Age(\bU)$ has the \AEPn and the \HAPn.
\end{proof}

\subsection{Sufficient condition for the existence of universal homogeneous polymorphisms}
Though, Theorem~\ref{univhompoly} gives necessary and sufficient conditions for countable  homogeneous relational  structures to have universal homogeneous polymorphisms, unfortunately, these conditions are relatively difficult to verify. The goal of this section is to give sufficient conditions for the existence of universal homogeneous polymorphisms, that are somewhat easier to test.
\begin{definition}
A class $\class{C}$ of $\uSigma$-structures is said to have the \emph{strict amalgamation property} if $\class{C}$ has the amalgamation property and if for all $\bA,\bB_1,\bB_2\in\class{C}$, and for all embeddings $f_1\colon   \bA\injto\bB_1$, $f_2\colon   \bA\injto\bB_2$ there exists some $\bC\in\class{C}$ and homomorphisms $g_1\colon   \bB_1\to\bC$, $g_2\colon   \bB_2\to\bC$ such that the following is a pushout-square in $(\class{C},\rightarrow)$:
        \begin{equation}\label{strictAP}
        \begin{tikzcd}
                \bB_1 \rar[dashed]{g_1}& \bC\\
                \bA\pushout\uar[hook]{f_1}\rar[hook]{f_2}&  \bB_2\uar[dashed]{g_2}.
        \end{tikzcd}
        \end{equation}
        An age that has the strict amalgamation property is called a \emph{strict \Fraisse-class}.
\end{definition}
\begin{remark}
    The homomorphisms $g_1$ and $g_2$ in diagram \eqref{strictAP} are automatically embeddings, because $\class{C}$ has the amalgamation property. If $f_1$, $f_2$, $g_1$, $g_2$  are identical embeddings, then the structure $\bC$ will be denoted by $\bB_1\oplus_\bA \bB_2$ and will be called the \emph{amalgamated free sum} of $\bB_1$ and $\bB_2$ with respect to $\bA$.   

	Note also that every  \Fraisse class that has the free amalgamation property is also a strict \Fraisse class.	
	Examples for strict \Fraisse classes without the free amalgamation property are given by the class of finite posets, the class of finite strict posets, and the class of non-empty metric spaces with rational distances.
\end{remark}

\begin{definition}
  Let $\class{C}$ be a class of $\uSigma$-structures closed under finite products and enjoying the strict amalgamation property.  We say that $\class{C}$ has
  \emph{well-behaved amalgamated free sums} if for all pushout-diagrams 
\[
  \begin{tikzcd}
    \bB_1 \rar{\kappa_{\bB_1}} & \bB_1\oplus_{\bA_1}\bC_1 & \bB_2 \rar{\kappa_{\bB_2}} & \bB_2\oplus_{\bA_2}\bC_2 \\
    \bA_1\pushout\uar[hook]{=}\rar[hook]{=}& \bC_1\uar{\kappa_{\bC_1}} & \bA_2\pushout\uar[hook]{=}\rar[hook]{=}& \bC_2\uar{\kappa_{\bC_2}} 
  \end{tikzcd}
\]
\[
  \begin{tikzcd}
    \bB_1\times\bB_2 \rar{\kappa_{\bB_1\times\bB_2}}&
    (\bB_1\times\bB_2)\oplus_{\bA_1\times\bA_2}(\bC_1\times\bC_2)\\
    \bA_1\times\bA_2\pushout\uar[hook]{=}\rar[hook]{=}& \bC_1\times\bC_2\uar{\kappa_{\bC_1\times\bC_2}}
  \end{tikzcd}
\]
in $(\class{C},\to)$, the unique homomorphism  $h\colon  (\bB_1\times\bB_2)\oplus_{\bA_1\times\bA_2}(\bC_1\times\bC_2)\to(\bB_1\oplus_{\bA_1}\bC_1)\times (\bB_2\oplus_{\bA_2}\bC_2)$ that makes
the following diagram commutative
\[
\begin{tikzcd}
  ~ & &\hspace*{-2.5cm} (\bB_1\oplus_{\bA_1}\bC_1)\times (\bB_2\oplus_{\bA_2}\bC_2)\\
\bB_1\times\bB_2 \rar{\kappa_{\bB_1\times\bB_2}}\arrow[bend
left]{urr}[sloped, near end]{\kappa_{\bB_1}\times\kappa_{\bB_2}} & (\bB_1\times\bB_2)\oplus_{\bA_1\times\bA_2}(\bC_1\times\bC_2)
\urar[dashed]{h}\\
\bA_1\times\bA_2
\uar[hook]{=}\rar[hook]{=}\pushout
& \bC_1\times\bC_2,\uar{\kappa_{\bC_1\times\bC_2}}
\arrow[bend right]{uur}[swap,sloped,near start]{\kappa_{\bC_1}\times\kappa_{\bC_2}}
\end{tikzcd}
\]
is an embedding. 
\end{definition}

\begin{lemma}\label{wellbeh}
    Let $\class{C}$ be a class of $\uSigma$-structures with the strict amalgamation property, that is closed under finite products. Suppose further that $\class{C}$ has well-behaved amalgamated free sums. Given a pushout square
    \[
    \begin{tikzcd}
    \bB_1\rar[hook]{g_1} & \bC \\
    \bA \pushout\uar[hook]{=}\rar[hook]{=}& \bB_2.\uar[hook]{g_2}
    \end{tikzcd}
    \]
    Consider the pushout square 
  \begin{equation*}\label{npushout}
  \begin{tikzcd}
    \bB_1^n \rar[hook]{\hat{g}_1}& \widehat\bC\\
    \bA^n\pushout\uar[hook]{=}\rar[hook,swap]{=}& \bB_2^n.\uar[hook]{\hat{g}_2}
  \end{tikzcd}
  \end{equation*}
    Then the unique mediating morphism $k\colon   \widehat\bC\to\bC^n$, that makes the following diagram commutative:
 \begin{equation}\label{nmediating}
\begin{tikzcd}
    ~& & \bC^n \\
    \bB_1^n \arrow[bend left,hook]{urr}{g_1^n}\rar[hook]{\hat{g}_1}& \widehat\bC\urar{k}\\
    \bA^n\pushout \uar[hook]{=}\rar[hook,swap]{=}& \bB_2^n \arrow[bend right,hook]{uur}[swap]{g_2^n}\uar[hook]{\hat{g}_2}
  \end{tikzcd}
  \end{equation}
    is an embedding. 
\end{lemma}
\begin{proof}
    We proceed by induction on $n$. The case $n=1$ is immediate.
    
    Suppose the claim is true for some given $n$. 
   By induction hypothesis, the unique mediation arrow $k$ in \eqref{nmediating} is an embedding.
 Since $\class{C}$ has well-behaved amalgamated free sums, the mediating arrow $\tilde{k}$ in the following diagram is an embedding, too:
 \[
      \begin{tikzcd}
    ~& & \widehat\bC\times\bC \\
    \bB_1^{n+1} \arrow[bend left,hook]{urr}{\hat{g}_1\times g_1}\rar[hook]{\tilde{g}_1}& \widetilde\bC\urar[hook,dashed]{\tilde{k}}\\
    \bA^{n+1}\pushout\uar[hook]{=}\rar[hook,swap]{=}& \bB_2^{n+1}.\arrow[bend
    right,hook]{uur}[swap]{\hat{g}_2\times g_2}\uar[hook]{\tilde{g}_2}
  \end{tikzcd}
  \]
  We conclude that then the following diagram commutes:
  \[
  \begin{tikzcd}
      ~ & & & \bC^{n+1}\\
      & & \widehat\bC\times\bC\urar[hook,swap,near start]{k\times 1_{\bC}}\\ 
    \bB_1^{n+1} \arrow[bend left,out=45,hook]{uurrr}{g_1^{n+1}}\arrow[bend left,hook]{urr}[near end]{\hat{g}_1\times g_1}\rar[hook]{\tilde{g}_1}& \widetilde\bC\urar[hook]{\tilde{k}}\\
    \bA^{n+1}\pushout \uar[hook]{=}\rar[hook,swap]{=}& \bB_2^{n+1}.\arrow[bend right,in=-135,hook]{uuurr}[swap]{g_2^{n+1}}\arrow[bend
    right,hook]{uur}[swap,near end]{\hat{g}_2\times g_2}\uar[hook]{\tilde{g}_2}
  \end{tikzcd}
\]
  Hence $k':= (k\times 1_{\bC})\circ\tilde{k}$ is the unique mediating morphism that makes the following diagram commutative:
   \[
      \begin{tikzcd}
    ~& & \bC^{n+1} \\
    \bB_1^{n+1} \arrow[bend left,hook]{urr}{g_1^{n+1}}\rar[hook]{\tilde{g}_1}& \widehat\bC\urar{k'}\\
    \bA^{n+1}\pushout \uar[hook]{=}\rar[hook,swap]{=}& \bB_2^{n+1}.\arrow[bend
    right,hook]{uur}[swap]{g_2^{n+1}}\uar[hook]{\tilde{g}_2}
  \end{tikzcd}
  \]
    Moreover, since both, $k\times 1_{\bC}$ and $\tilde{k}$ are embeddings, we have that $k'$ is an embedding, too.
\end{proof}

\begin{proposition}\label{UH6condAlg}
  Let $\class{C}$ be a class of $\uSigma$-structures such that
  \begin{enumerate}
  \item $\class{C}$ has the strict amalgamation property,
  \item $\class{C}$ is closed with respect to finite products,
  \item $\class{C}$ has well-behaved amalgamated free sums,
  \item $\class{C}$ has the \HAP.
  \end{enumerate}
  Then $\class{C}$ has the \AEPn, for every $n\in\bN\setminus\{0\}$.
\end{proposition}
\begin{proof}
  Let $\bA,\bB_i, \bT\in\class{C}$, $f_i\colon   \bA\injto\bB_i$, $h_i\colon   \bB_i^n\to \bT$ (where $i\in\{1,2\}$), with $h_1\circ f_1^n= h_2\circ  f_2^n$.
    
  Let $\bC\in\class{C}$, $g_1\colon   \bB_1\injto \bC$, $g_2\colon   \bB_2\injto \bC$ such that the following is a pushout-square in $(\class{C},\to)$:
  \[
  \begin{tikzcd}
    \bB_1 \rar[hook]{g_1}& \bC\\
    \bA\pushout \uar[hook]{f_1}\rar[hook]{f_2}&
    \bB_2.\uar[hook]{g_2}
  \end{tikzcd}
  \]
  Since $\class{C}$ is closed with respect to finite products, $\bA^n$, $\bB_1^n$, $\bB_2^n$ are in $\class{C}$. Since $\class{C}$ has the strict amalgamation property, there exists $\widehat\bC\in\class{C}$, $\hat{g}_1\colon   \bB_1^n\injto\widehat\bC$, $\hat{g}_2\colon   \bB_2^n\injto\widehat\bC$ such that the following is a pushout-square in $(\class{C},\to)$:
  \[
  \begin{tikzcd}
    \bB_1^n \rar[hook]{\hat{g}_1}& \widehat\bC\\
    \bA^n\pushout \uar[hook]{f_1^n}\rar[hook,swap]{
      f_2^n}& \bB_2^n.\uar[hook]{\hat{g}_2}
  \end{tikzcd}
  \]
  Hence, there exists $k\colon   \widehat\bC\to\bC^n$ such that the following diagram commutes:
  \begin{equation}\label{star3}
  \begin{tikzcd}
    ~& & \bC^n \\
    \bB_1^n \arrow[bend left,hook]{urr}{g_1^n}\rar[hook]{\hat{g}_1}& \widehat\bC\urar{k}\\
    \bA^n \uar[hook]{f_1^n}\rar[hook,swap]{
      f_2^n}& \bB_2^n.\arrow[bend
    right,hook]{uur}[swap]{g_2^n}\uar[hook]{\hat{g}_2}
  \end{tikzcd}
  \end{equation}
  Moreover, by Lemma~\ref{wellbeh}, $k$ is an embedding.
  
  Next we note that there exists $h\colon \widehat\bC\to\bT$ such that the following diagram commutes:
  \begin{equation}\label{star1}
  \begin{tikzcd}
    ~& & \bT \\
    \bB_1^n \arrow[bend left]{urr}{h_1}\rar[hook]{\hat{g}_1}& \widehat\bC\urar{h}\\
    \bA^n \uar[hook]{f_1^n}\rar[hook,swap]{
      f_2^n}& \bB_2^n.\arrow[bend
    right]{uur}[swap]{h_2}\uar[hook]{\hat{g}_2}
  \end{tikzcd}
  \end{equation}
  Since $\class{C}$ has the \HAP, there exist $\hat{k}\colon   \bT\injto \bT'$, and a homomorphism $\hat{h}\colon \bC^n\to\bT'$ such that the following diagram commutes:
  \begin{equation}\label{star2}
  \begin{tikzcd}
    \bC^n\rar[dashed]{\hat{h}} & \bT'\\
    \widehat{\bC} \uar[hook]{k}\rar{h}& \bT\uar[hook,dashed]{\hat{k}}.
  \end{tikzcd}
  \end{equation}
It remains to observe that the following diagram commutes:
 \begin{equation}\label{star4}
  \begin{tikzcd}
     & & \bT\rar[hook][swap]{\hat{k}}&\bT'\\
     \bB_1^n
    \arrow[bend left]{urr}{h_1} \rar[hook]{g_1^n}
    & \bC^n\arrow[bend right]{urr}[swap,near end]{\hat{h}}\\
    \bA^n
    \uar[hook]{f_1^n}\rar[hook]{f_2^n} & 
    \bB_2^n. \arrow[bend right]{uur}[swap]{h_2}
    \uar[swap,hook]{g_2^n}
  \end{tikzcd}
  \end{equation}
  Indeed, we compute:
  \begin{equation*}
      \hat{k}\circ h_1 \stackrel{\eqref{star1}}{=} \hat{k}\circ h\circ\hat{g}_1\stackrel{\eqref{star2}}{=} \hat{h}\circ k\circ\hat{g}_1\stackrel{\eqref{star3}}{=} \hat{h}\circ g_1^n.
  \end{equation*}
  Analogously the identity $\hat{k}\circ h_2=\hat{h}\circ g_2^n$ may be shown. From these two identities it follows that diagram \eqref{star4} commutes. Hence $\class{C}$ has the \AEPn, for every $n\in\bN\setminus\{0\}$. 
\end{proof}

\section{Structures with universal homogeneous polymorphisms}
\subsection{Free homogeneous structures} 

    Let $\uSigma$ be a relational signature and let $\cat{C}_\uSigma$ be the category of all $\uSigma$-structures with homomorphisms as morphisms.
\begin{lemma}\label{weakpushouts}
    Let $n\in\bN\setminus\{0\}$, and for each $i\in\{1,2\}$, let $\bA,\bB_i,\bC\in\cat{C}_\uSigma$, $f_i\colon \bA\injto\bB_i$, $g_i\colon \bB_i\injto\bC$, such that the following is a pushout-square in $\cat{C}_\uSigma$:
  \[
  \begin{tikzcd}
    \bB_1 \rar[hook]{g_1}& \bC \\
    \bA\pushout\uar[hook]{f_1}\rar[hook]{f_2} & \bB_2.\uar[hook]{g_2}
  \end{tikzcd}
  \]
  Then the following is a weak pushout-square in $\cat{C}_\uSigma$:
  \[
  \begin{tikzcd}
    \bB_1^n \rar[hook]{g_1^n}& \bC^n \\
    \bA^n\uar[hook]{f_1^n}\rar[hook]{f_2^n} & \bB_2^n.\uar[hook]{g_2^n}
  \end{tikzcd}
  \]
\end{lemma}
\begin{proof}
  Let $\widehat\bC\in\cat{C}_\uSigma$, $\hat{g_i}\colon \bB_{i}^n\injto\widehat\bC$ (for $i\in\{1,2\}$), such that the following is a pushout-square in $\cat{C}_\uSigma$.
  \[
  \begin{tikzcd}
    \bB_1^n \rar[hook]{\hat g_1}& \widehat\bC \\
    \bA^n\pushout\uar[hook]{f_1^n}\rar[hook]{f_2^n} & \bB_2^n.\uar[hook]{\hat
      g_2}
  \end{tikzcd}
  \]
  It remains to construct a homomorphism $h\colon \bC^n\to\widehat\bC$ such that the following diagram commutes:
  \begin{equation}\label{diagpushout}
    \begin{tikzcd}
      ~& & \widehat\bC\\
      \bB_1^n \arrow[bend left]{urr}{\hat{g}_1}\rar[hook]{g_1^n}& \bC^n\urar[dashed]{h} \\
      \bA^n\uar[hook]{f_1^n}\rar[hook]{f_2^n} & \bB_2^n.\arrow[bend
      right]{uur}[swap]{\hat{g}_2}\uar[hook]{g_2^n}
    \end{tikzcd}
  \end{equation} 
  We define
  \[
  h(x_1,\dots,x_n):=\begin{cases}
    \hat{g}_1(u_1,\dots,u_n) & (x_1,\dots,x_n)=(g_1(u_1),\dots,g_1(u_n))\\
    \hat{g}_2(v_1,\dots,v_n) & (x_1,\dots,x_n)=(g_2(v_1),\dots,g_2(v_n))\\
    \hat{g}_1(u_1,\dots,u_1) & \text{else, if }g_1(u_1)=x_1\\
    \hat{g}_2(v_1,\dots,v_1) & \text{else, if }g_2(v_1)=x_1.
  \end{cases}
  \]
  It remains to show that $h$ is well-defined and a homomorphism. Suppose, that
  \[(g_2(v_1),\dots,g_2(v_n))=(x_1,\dots,x_n)=(g_1(u_1),\dots,g_1(u_n)).\]
  Since $\bC$ is the free amalgamated sum of $g_1(\bB_1)$ with $g_2(\bB_2)$ with respect to $g_1(f_1(\bA))$, there exist $(a_1,\dots,a_n)\in A^n$, such that $(f_1(a_1),\dots,f_1(a_n))=(u_1,\dots,u_n)$ and $(f_2(a_1),\dots,f_2(a_n))=(v_1,\dots,v_n)$. But since $\hat{g}_1\circ f_1^n=\hat{g}_2\circ f_2^n$, we obtain
  \begin{equation*}
    \hat{g}_1(u_1,\dots,u_n)=\hat{g}_1(f_1(a_1),\dots, f_1(a_n))= \hat{g}_2(f_2(a_1),\dots, f_2(a_n))= \hat{g}_2(v_1,\dots,v_n).
  \end{equation*}
  If neither 
  $(x_1,\dots,x_n)=(g_1(u_1),\dots,g_1(u_n))$, nor $(x_1,\dots,x_n)=(g_2(v_1),\dots,g_2(v_n))$,
   but $g_1(u_1)=x_1=g_2(v_1)$, then, since $g_1(B_1)\cap g_2(B_2)=g_1(f_1(A))=g_2(f_2(A))$, there exists $a_1\in A$ such that $f_1(a_1)=u_1$, and $f_2(a_1)=v_1$. Hence,
  \begin{equation*}
    \hat{g}_1(u_1,\dots,u_1)=\hat{g}_1(f_1(a_1),\dots, f_1(a_1))= \hat{g}_2(f_2(a_1),\dots, f_2(a_1))= \hat{g}_2(v_1,\dots,v_1).
  \end{equation*}
  Thus, $h$ is well-defined.
    
  Let $\varrho$ be a relational symbol of arity $m$ from $\Sigma$, and let $(\ta_1,\dots,\ta_m)\in\varrho^{\bC^n}$, where 
  \[\ta_i=(a_{i,1},\dots,a_{i,n}) \text{ (for $i\in\{1,\dots,m\}$)}.\]
   Then  we have that $(a_{1,j},\dots,a_{m,j})$ is in $\varrho^\bC$, for each  $j\in\{1,\dots,n\}$. Since $\varrho^\bC=g_1(\varrho^{\bB_1})\cup g_2(\varrho^{\bB_2})$, for every $j\in\{1,\dots,n\}$ we have $(a_{1,j},\dots,a_{m,j})\in g_1(\varrho^{\bB_1})$ or $(a_{1,j},\dots,a_{m,j})\in g_2(\varrho^{\bB_2})$.
    
  Suppose that for every $j\in\{1,\dots,n\}$ there exists $(u_{1,j},\dots,u_{m,j})\in\varrho^{\bB_1}$, such that 
  \[(a_{1,j},\dots,a_{m,j})=(g_1(u_{1,j}),\dots,g_1(u_{m,j})),\]
   then we have
  \[ \left[\begin{array}{c}
      h(a_{1,1},\dots,a_{1,n})\\
      \vdots\\
      h(a_{m,1},\dots,a_{m,n})
    \end{array}\right]=\left[\begin{array}{c}
      \hat{g}_1(u_{1,1},\dots,u_{1,n})\\
      \vdots\\
      \hat{g}_1(u_{m,1},\dots,u_{m,n})
    \end{array}\right]\in\varrho^{\widehat{\bC}},
  \]
  since $\hat{g}_1$ is a homomorphism.

  Analogously, if for every $j\in\{1,\dots,n\}$ there exists $(v_{1,j},\dots,v_{m,j})\in\varrho^{\bB_2}$, such that
  \[(a_{1,j},\dots,a_{m,j})=(g_2(v_{1,j}),\dots,g_2(v_{m,j})),\]
   then we have
  \[ \left[\begin{array}{c}
      h(a_{1,1},\dots,a_{1,n})\\
      \vdots\\
      h(a_{m,1},\dots,a_{m,n})
    \end{array}\right]=\left[\begin{array}{c}
      \hat{g}_2(v_{1,1},\dots,v_{1,n})\\
      \vdots\\
      \hat{g}_2(v_{m,1},\dots,v_{m,n})
    \end{array}\right]\in\varrho^{\widehat{\bC}},
  \]
  since $\hat{g}_2$ is a homomorphism.

  Otherwise, if there exists $(u_1,\dots,u_m)\in\varrho^{\bB_1}$, such that 
  \[(a_{1,1},\dots,a_{m,1})= (g_1(u_1),\dots,g_1(u_m)),\]
  then
  \[ \left[\begin{array}{c}
      h(a_{1,1},\dots,a_{1,n})\\
      \vdots\\
      h(a_{m,1},\dots,a_{m,n})
    \end{array}\right]=\left[\begin{array}{c}
      \hat{g}_1(u_1,\dots,u_1)\\
      \vdots\\
      \hat{g}_1(u_m,\dots,u_m)
    \end{array}\right]\in\varrho^{\widehat{\bC}},
  \]
  and if there exists $(v_1,\dots,v_m)\in\varrho^{\bB_2}$, such that $(a_{1,1},\dots,a_{m,1})= (g_2(v_1),\dots,g_2(v_m))$, then
  \[ \left[\begin{array}{c}
      h(a_{1,1},\dots,a_{1,n})\\
      \vdots\\
      h(a_{m,1},\dots,a_{m,n})
    \end{array}\right]=\left[\begin{array}{c}
      \hat{g}_2(v_1,\dots,v_1)\\
      \vdots\\
      \hat{g}_2(v_m,\dots,v_m)
    \end{array}\right]\in\varrho^{\widehat{\bC}}.
  \]
  Thus, $h$ is a homomorphism.

  By construction of $h$ we have that diagram~\eqref{diagpushout} commutes. Thus, the proof is complete.
\end{proof}

\begin{proposition}\label{freeaepn}
  Let $\bU$ be a countably infinite homogeneous relational structure whose age has the free amalgamation property.  Then $\Age(\bU)$ has the $\AEPn$, for every $n\in\bN\setminus\{0\}$.
\end{proposition}
\begin{proof}
   Let $n\in\bN\setminus\{0\}$. 
    
   Given $\bA, \bB_1, \bB_2,\bT\in\Age(\bU)$, $f_1\colon   \bA\injto\bB_1$, $f_2\colon   \bA\injto\bB_2$, $h_1\colon   \bB_1\to\bT$, $h_2\colon   \bB_2\to\bT$, such that $h_1\circ f_1=h_2\circ f_2$. Without loss of generality, $f_1$ and $f_2$ are identical embeddings and $B_1\cap B_2=A$.  

  Let $\bC:=\bB_1\oplus_{\bA}\bB_2$, in other words, the following is a pushout-square in $\cat{C}_\uSigma$:
  \begin{equation*}
  \begin{tikzcd}
    \bB_1 \rar[hook]{=}& \bC \\
    \bA\pushout  \uar[hook]{=}\rar[hook]{=}&\bB_2\uar[hook]{=}.
  \end{tikzcd}
  \end{equation*}
  By Lemma~\ref{weakpushouts}, the following is a weak pushout square in $\cat{C}_\uSigma$:
\[
  \begin{tikzcd}
    \bB_1^n \rar[hook]{=}& \bC^n \\
    \bA^n  \uar[hook]{=}\rar[hook]{=}&\bB_2^n\uar[hook]{=}.
  \end{tikzcd}
\]
Hence there exists some $h\colon   \bC^n\to \bT$ such that the following diagram commutes:
  \[
  \begin{tikzcd}
    ~& & \bT \\
    \bB_1^n \arrow[bend left]{urr}{h_1}\rar[hook]{=}& \bC^n\urar{h}\\
    \bA^n \uar[hook]{=}\rar[hook,swap]{=}& \bB_2^n\arrow[bend
    right]{uur}[swap]{h_2}\uar[hook]{=}
  \end{tikzcd}
  \]
    Taking $\bT':=\bT$, we obtain, that the following diagram commutes, too:
 \[
  \begin{tikzcd}
     & & \bT\rar[hook][swap]{=}& \bT'\\
     \bB_1^n
    \arrow[bend left]{urr}{h_1} \rar[hook]{=}
    & \bC^n\arrow[bend right]{urr}[swap,near end]{h}\\
    \bA^n
    \uar[hook]{=}\rar[hook]{=} & 
    \bB_2^n. \arrow[bend right]{uur}[swap]{h_2}
    \uar[swap,hook]{=}
  \end{tikzcd}
  \]
    Thus, $\Age(\bU)$  has the \AEPn. 
\end{proof}

\begin{corollary}\label{univpolyfree}
  Let $\bU$ be a countably infinite homogeneous relational structure whose age has the free amalgamation property. Let $n\in\bN\setminus\{0\}$. Then $\bU$ has an $n$-ary universal homogeneous polymorphism if and only if $\Age(\bU)$ has the $\HAPn$. 
\end{corollary}
\begin{proof}
    This follows directly from Proposition~\ref{freeaepn}, in conjunction with Theorem~\ref{univhompoly}.
\end{proof}

\begin{example}
  The Rado-graph has universal homogeneous polymorphisms  of every arity, since its age is closed with respect to finite products, has the \HAP, and has the free amalgamation property. 
  
  For the same reasons, the countable universal homogeneous digraph and the countable universal homogeneous $k$-hypergraphs have universal homogeneous polymorphisms of all arities.
\end{example}

\subsection{The generic poset} \label{secGenPos} We are going to consider the countable generic poset $\mathbb{P}$ with both, the strict and the non-strict ordering.

The following construction of amalgamated free sums in the category of posets is folklore:
\begin{construction}
  Let $\bA$, $\bB_1$, $\bB_2$ be posets such that $\bA\le \bB_1$, $\bA\le\bB_2$, and such that $B_1\cap B_2 = A$. Define $C:= B_1\cup B_2$, $(\le_\bC):=(\le_{\bB_1})\cup (\le_{\bB_2})\cup\sigma\cup\tau$,  where
  \begin{align*}
    \sigma &= \{(b_1,b_2)\mid b_1\in B_1, b_2\in B_2,\exists a\in A:
    b_1\le_{\bB_1} a\le_{\bB_2} b_2\},\\
    \tau &=\{(b_2,b_1)\mid b_1\in B_1, b_2\in B_2,\exists a\in A:
    b_2\le_{\bB_2} a\le_{\bB_1} b_1\},
  \end{align*}
  and finally $\bC:=(C,\le_{\bC})$. Then $\bC= \bB_1\oplus_\bA \bB_2$. In particular, the following is a pushout-square in the category of posets:
  \[
  \begin{tikzcd}
    \bB_1 \rar[hook]{=}& \bC \\
    \bA  \pushout\uar[hook]{=}\rar[hook]{=}&\bB_2\uar[hook]{=}.
  \end{tikzcd}
  \]
  The construction for the amalgamated free sums of strict posets is completely analogous to the above given construction. We just need to replace every occurrence of $\le$ through $<$.
\end{construction}

\begin{lemma}\label{posetwellbehaved}
  The class of finite posets has well-behaved amalgamated free sums. The same is true for the class of finite strict posets.
\end{lemma}
\begin{proof}
\textbf{The case of finite posets:}
  Given finite posets $\bA_1$, $\bB_{1,1}$, $\bB_{1,2}$, $\bA_2$, $\bB_{2,1}$, $\bB_{2,2}$, such that $\bA_1\le\bB_{1,1}$, $\bA_1\le\bB_{1,2}$, $B_{1,1}\cap B_{1,2}=A_1$, $\bA_2\le\bB_{2,1}$, $\bA_2\le\bB_{2,2}$, $B_{2,1}\cap B_{2,2}=A_2$.
  
  Let $\bC_1:=\bB_{1,1}\oplus_{\bA_1}\bB_{1,2}$, $\bC_2:=\bB_{2,1}\oplus_{\bA_2}\bB_{2,2}$, and let $\bD:=(\bB_{1,1}\times\bB_{2,1})\oplus_{\bA_1\times\bA_2}(\bB_{1,2}\times\bB_{2,2})$. We will show that $\bD\le\bC_1\times\bC_2$.

  First we note
  \begin{align*}
    D&=
    B_{1,1}\times B_{2,1}\cup B_{1,2}\times B_{2,2} 
    \subseteq B_{1,1}\times B_{2,1}\cup B_{1,1}\times B_{2,2}\cup
    B_{1,2}\times B_{2,1}\cup B_{1,2}\times B_{2,2,} = C_1\times C_2.
  \end{align*}

  Now we will show that $(\le_{\bD}) = (\le_{\bC_1\times\bC_2})\cap D^2$.

  ``$\subseteq$:'' Let $(u_1,u_2), (v_1,v_2)\in D$, such that $(u_1,u_2)\le_{\bD}(v_1,v_2)$. If $(u_1,u_2),(v_1,v_2)\in B_{1,1}\times B_{2,1}$, then
  \begin{align*}
    (u_1,u_2)\le_{\bD}(v_1,v_2) &\iff
    (u_1,u_2)\le_{\bB_{1,1}\times\bB_{2,1}}(v_1,v_2)
    \iff u_1\le_{\bB_{1,1}} v_1 \land u_2\le_{\bB_{2,1}} v_2\\
    &\iff u_1\le_{\bC_1} v_1 \land u_2\le_{\bC_2} v_2
   \iff (u_1,u_2)\le_{\bC_1\times\bC_2}(v_1,v_2).
  \end{align*}
  Analogously, if $(u_1,u_2),(v_1,v_2)\in B_{1,2}\times B_{2,2}$, then $(u_1,u_2)\le_{\bD}(v_1,v_2)$ if and only if $(u_1,u_2)\le_{\bC_1\times\bC_2}(v_1,v_2)$.

  Suppose that $(u_1,u_2)\in B_{1,1}\times B_{2,1}$, $(v_1,v_2)\in B_{1,2}\times B_{2,2}$. Then
  \begin{align*}
        (u_1,u_2)\le_{\bD}(v_1,v_2) &\iff
         \exists (a_1,a_2)\in A_1\times A_2:
        (u_1,u_2)\le_{\bB_{1,1}\times\bB_{2,1}}
        (a_1,a_2)\le_{\bB_{1,2}\times\bB_{2,2}}(v_1,v_2)\\
        &\iff \exists (a_1,a_2)\in A_1\times A_2: u_1\le_{\bB_{1,1}} a_1 \le_{\bB_{1,2}} v_1 \land
        u_2\le_{\bB_{2,1}} a_2 \le_{\bB_{2,2}} v_2\\
        &\iff u_1\le_{\bC_1} v_1 \land u_2\le_{\bC_2} v_2\iff (u_1,u_2)\le_{\bC_1\times\bC_2}(v_1,v_2).
  \end{align*}
Analogously the case $(u_1,u_2)\in B_{1,2}\times B_{2,2}$, $(v_1,v_2)\in B_{1,1}\times B_{2,1}$ is handled.  

	\textbf{The case of finite strict posets:} This case is analogous to the previous one. As before, we only need to replace all occurrences of $\le$ by $<$.
\end{proof}

\begin{lemma}\label{posHAP}
    The classes of finite posets and of finite strict posets both have the HAP.
\end{lemma}
\begin{proof}
	In \cite{CamLoc10} and \cite{Mas07} the homomorphism homogeneous countable posets and strict posets are completely classified. From these classifications it can be read off that both, $(\mathbb{P},\le)$ and $(\mathbb{P},<)$, are homomorphism homogeneous. From \cite[Proposition 3.8]{Dol14} it follows that  their ages both have the \HAP. 
\end{proof}

\begin{corollary}\label{posnHAPn}
    The class of finite (strict) posets has the  \HAPn, for every $n\in\bN\setminus\{0\}$. 
\end{corollary}
\begin{proof}
    This follows directly from the fact that the class of  finite (strict) posets is closed under finite products and has the \HAP (cf. Lemma~\ref{posHAP}).
\end{proof}

\begin{corollary}\label{posnAEPn}
    The class of finite (strict) posets has the \AEPn, for every $n\in\bN\setminus\{0\}$. 
\end{corollary}
\begin{proof}
    This follows directly from Lemmas~\ref{posetwellbehaved}, \ref{posHAP} in conjunction with Proposition~\ref{UH6condAlg}.
\end{proof}

\begin{theorem}\label{genposetunivpoly}
  The generic posets  $(\mathbb{P},\le)$ and $(\mathbb{P},<)$ have  universal homogeneous polymorphisms of every arity.
\end{theorem}
\begin{proof}
     By Corollaries~\ref{posnAEPn}, \ref{posnHAPn} we have that both, $\Age(\mathbb{P},\le)$ and $\Age(\mathbb{P},<)$, have the \AEPn, and the \HAPn, for every $n\in\bN\setminus\{0\}$. Finally, by   Theorem~\ref{univhompoly}, $(\mathbb{P},<)$  and $(\mathbb{P},\le)$ have universal homogeneous polymorphisms of every arity.
\end{proof}

\section{Clones with automatic homeomorphicity}

\begin{theorem}
	Let $\bU$ be a countable homogeneous relational structure such that 
	\begin{enumerate}
		\item $\Pol(\bU)$ contains all constant functions,
		\item $\Age(\bU)$ has the free amalgamation property,
		\item $\Age(\bU)$ is closed with respect to finite products,
		\item $\Age(\bU)$ has the \HAP.
	\end{enumerate}
	Then $\Pol(\bU)$ has automatic homeomorphicity.
\end{theorem}
\begin{proof}
	Let $h$ be an isomorphism of $\Pol(\bU)$ to the polymorphism clone of another countable structure. Since $\Pol(\bU)$ contains all constant functions, it follows from Proposition~\ref{bbpopen}, that $h$ is open. Since $\Age(\bU)$ has the \HAP, and  is closed with respect to finite products, it follows that it has the \HAPn, for all $n\in\bN\setminus\{0\}$. Thus, since $\Age(\bU)$ has the free amalgamation property, it follows from Corollary~\ref{univpolyfree}, that $\bU$ has universal homogeneous polymorphisms of all arities. Thus, by Proposition~\ref{stronggateexistence}, it follows that $\Pol(\bU)$ has a strong gate covering. Since $h$ is open, it follows form Proposition~\ref{autocontcriterion}, that $h$ is a homeomorphism. 
\end{proof}

\begin{corollary}
	The polymorphism clones of the following structures have automatic homeomorphicity:
	\begin{itemize}
		\item the structure $(\bN,=)$ (shown already in \cite[Corollary 28]{BodPinPon17}),
		\item the Rado graph with all loops added,
		\item the universal homogeneous digraph with all loops added,
	\end{itemize}
\end{corollary}

\begin{theorem}
    The polymorphism clone of the generic poset $(\mathbb{P},\le)$ has automatic homeomorphicity.
\end{theorem}
\begin{proof}
    Let $h$ be an isomorphism from $\Pol(\mathbb{P},\le)$ to the polymorphism clone of another countable  structure. 
    
    Clearly, all constant functions are polymorphisms of $(\mathbb{P},\le)$. Thus, by Proposition~\ref{bbpopen},  $h$ is open.

    By Theorem~\ref{genposetunivpoly}, $(\mathbb{P},\le)$ has universal homogeneous polymorphisms of all arities. By Proposition~\ref{stronggateexistence}, $\Pol(\mathbb{P},\le)$ has a strong gate covering. 
    Since $h$ is open, by Proposition~\ref{autocontcriterion}, $h$ is a homeomorphism. 
\end{proof}

\begin{theorem}\label{thstrat2}
	Let $\bU$ be a countable $\omega$-categorical homogeneous relational structure and let $\class{K}$ be a set of structures on $U$, containing $\bU$. Suppose that
	\begin{enumerate}
		\item $\Aut(\bU)$ acts transitively on $U$,
		\item $\overline{\Aut(\bU)}$ has automatic homeomorphicity with respect to  $\class{K}$,
		\item $\Age(\bU)$ has the free amalgamation property,
		\item $\Age(\bU)$ is closed with respect to finite products,
		\item $\Age(\bU)$ has the \HAP.
	\end{enumerate}
	Then $\Pol(\bU)$ has automatic homeomorphicity with respect  to $\class{K}$. 
\end{theorem}
In the proof we are going to make use of the following auxiliary result:
\begin{lemma}[{\cite[Lemma 4.1]{BehTruVar17}}]\label{closed}
	Let $\bU$ be a countable relational structure such that $\overline{\Aut(\bU)}$ has automatic homeomorphicity. Let $h\colon \End(\bU)\to \monoid{M}$ be a monoid-isomorphism to another closed transformation monoid on $U$. Then $h(\overline{\Aut(\bU)})$ is closed in $\monoid{M}$ and the restriction of $h$ to $\overline{\Aut(\bU)}$ is a topological embedding.
\end{lemma}

\begin{proof}[Proof of Theorem~\ref{thstrat2}]
	Let $h$ be an isomorphism from $\Pol(\bU)$ to the polymorphism clone of a member of $\class{K}$. Since $\Age(\bU)$ has the \HAP, and  is closed with respect to finite products, it follows that it has the \HAPn, for all $n\in\bN\setminus\{0\}$. Thus, since $\Age(\bU)$ has the free amalgamation property, it follows from Corollary~\ref{univpolyfree}, that $\bU$ has universal homogeneous polymorphisms of all arities. Thus, by Proposition~\ref{stronggateexistence}, it follows that $\Pol(\bU)$ has a strong gate covering. Since $\overline{\Aut(\bU)}$ has automatic homeomorphicity, it follows from Lemma~\ref{closed} that the restriction of $h$ to $\overline{\Aut(\bU)}$ is a topological embedding. In particular, $h\restr_{\overline{\Aut(\bU)}}$ is continuous. Consequently, since $\Pol(\bU)$ has a strong gate covering, it follows from Lemma~\ref{liftcont} that $h$ is continuous, too. Thus, since $\bU$ is $\omega$-categorical, $\Aut(\bU)$ acts transitively on $U$, $\Age(\bU)$ has the free amalgamation property, $\Age(\bU)$ is closed with respect to finite products, and $\Age(\bU)$ has the \HAP, it follows from Proposition~\ref{autohomhap} that $h$ is a homeomorphism.
\end{proof}
\begin{corollary}
	Let $\bU$ be a countable $\omega$-categorical homogeneous relational structure and let $\class{K}$ be a set of structures on $U$, containing $\bU$. Suppose that
	\begin{enumerate}
		\item $\Aut(\bU)$ acts transitively on $U$, but $
		\Aut(\bU)$ is not the full symmetric group,
		\item $\Aut(\bU)$ has automatic homeomorphicity with respect to  $\class{K}$,
		\item $\Age(\bU)$ has the free amalgamation property,
		\item $\Age(\bU)$ is closed with respect to finite products,
		\item $\Age(\bU)$ has the \HAP.
	\end{enumerate}
	Then $\Pol(\bU)$ has automatic homeomorphicity with respect  to $\class{K}$. 
\end{corollary}
\begin{proof}
	We only need to show that $\overline{\Aut(\bU)}$ has automatic homeomorphicity with respect to $\class{K}$, in order to be able to invoke Theorem~\ref{thstrat2}: Since $\bU$ is $\omega$-categorical, it follows that $\bU$ is $\omega$-saturated. Since $\Aut(\bU)$ is transitive and $\Age(\bU)$ has the free amalgamation property, and since $\Aut(\bU)$ is not the full symmetric group on $U$, we conclude using \cite[Theorem 4.2.7]{Mac11} that $\Aut(\bU)$ is simple. In particular, it has a trivial center. Now, from Corollary~\ref{newmonoidauthom} it follows, that $\overline{\Aut(\bU)}$ has automatic homeomorphicity with respect to $\class{K}$.
\end{proof}

\begin{example}
	The polymorphism clones of the following countably infinite structures have automatic homeomorphicity:
	\begin{itemize}
		\item the Rado-graph (shown already in \cite[Theorem 52]{BodPinPon17}),
		\item the universal homogeneous digraph,
		\item the universal homogeneous $k$-uniform hypergraph (for all $k\ge 2$),
	\end{itemize}
\end{example}

Our last result concerns once more the generic poset:
\begin{theorem}\label{genposauthom}
	The polymorphism clone of the generic poset $(\mathbb{P},<)$ has automatic homeomorphicity with respect to the class of countable $\omega$-categorical structures.
\end{theorem}
Before we can prove this result, we need to adapt Proposition~\ref{autoopencriterion} to the case of the generic poset:
\begin{proposition}\label{genposautoopen}
	Let $\bU$ be a countable structure. Then every continuous isomorphism from $\Pol(\mathbb{P},<)$ to $\Pol(\bU)$ is a homeomorphism.
\end{proposition}
\begin{proof}
	The proof is virtually identical to the proof of Proposition~\ref{autoopencriterion}, and we are not going to repeat it in detail. Let $\algebra{S}\le(\mathbb{P},\Pol(\mathbb{P},<))^n$, and let $\sim$ be a congruence relation of $\algebra{S}$ with at least two classes. We are going to show that there exists some $i\in\{1,\dots,n\}$ such that $\tu\sim\tv\Rightarrow u_i=v_i$ holds for all $\tu,\tv\in \algebra{S}$. Once we succeed to show this, the rest of the proof is identical to the proof of Proposition~\ref{autoopencriterion}.
	
	By \cite[Theorem 6.29]{PecPec15}, $(\mathbb{P},<)$ is polymorphism homogeneous. Thus, by \cite[Corollary 3.13]{PecPec15}, has quantifier elimination for primitive positive formulae. Combining this with \cite[Theorem 4]{BodNes06}, we obtain that every invariant relation of $\Pol(\mathbb{P},<)$ is definable in $(\mathbb{P},<)$ by a set of atomic formulae.
	
	The carrier $S$ of $\algebra{S}$ is an $n$-ary invariant relation of $\Pol(\mathbb{P},<)$. The relation $\sigma^\sim=\{\tu\tv\mid\tu\sim\tv\}$ is a $2n$-ary invariant relation of $\Pol(\mathbb{P},<)$. Without loss of generality, $S$ contains at least one irreflexive tuple. Let $\Phi:=\Tpp_{(\mathbb{P},<)}(\sigma^\sim)$ and let $\Psi:= \Tpp_{(\mathbb{P},<)}(S)$. Since $\sim$ has at least two classes, in $\Phi$ there must exist $i,j\in\{1,\dots,n\}$ such that at least one of the atoms $x_i=y_j$, $x_i<y_j$, or $y_i<x_j$ is in $\Phi$. We are going to show that all atoms in $\Phi$ are of the shape $x_i=y_i$, by ruling out all other possibilities:
	
	 If $\Phi$ contains an atom of the shape $x_i=y_j$ for $i\neq j$, then, because $\sim$ is reflexive, we have that $x_i=x_j$ is in $\Psi$, a contradiction to our assumption about $S$. 
	 
	 Suppose $\Phi$ contains an atom of the shape $x_i<y_j$. Then $i\neq j$, since otherwise, by reflexivity of $\sim$ it would follow that $x_i<x_i$ is in $\Psi$. Without loss of generality, we can assume that there is no $k\in\{1,\dots,n\}$ such that $x_i<x_k\land x_k<y_j$ is in $\Phi$. By Lemma~\ref{oneGeneratedRel}, there exist $\tu,\tv\in S$ such that $\tu\tv\in\sigma^\sim$, and such that $\Tpp_{(\mathbb{P},<)}(\tu\tv)=\Phi$. Moreover, $\Tpp_{(\mathbb{P},<)}(\tu)=\Tpp_{(\mathbb{P},<)}(\tv)=\Psi$. Let $\bU:=\langle u_1,\dots,u_n\rangle_{(\mathbb{P},<)}$, $\bW:=\langle u_1,\dots,u_n,v_1,\dots,v_n\rangle_{(\mathbb{P},<)}$. Let $\bW'$ be an isomorphic copy of $\bW$, such that $W'=U\cup\{v_1',\dots,v_n'\}$, $W\cap W'=U$, and such that $\iota\colon \bW\to\bW'$ defined by $u_i\mapsto u_i$, $v_i\mapsto v_i'$ ($i\in\{1,\dots,n\})$ is an isomorphism. Consider $\bW\oplus_{\bU}\bW'$ (cf. Section~\ref{secGenPos}). Without loss of generality, $\bW\oplus_\bU\bW'\le(\mathbb{P},<)$. Let $\tv':=(v_1',\dots,v_n')$. By construction we have $\Tpp^{(0)}_{(\mathbb{P},<)}(\tu\tv)=\Tpp^{(0)}_{(\mathbb{P},<)}(\tu\tv')$. Since $(\mathbb{P},<)$ has quantifier elimination for primitive positive formulae, it follows that $\Tpp_{(\mathbb{P},<)}(\tu\tv)=\Tpp_{(\mathbb{P},<)}(\tu\tv')$. It follows that $\tu\sim\tv'$. Since $\sim$ is symmetric and transitive, we have $\tv\sim\tv'$. In particular, $v_i<v_j'$. By the construction of amalgamated free sums there exists some $k\in\{1,\dots,n\}$ such that $v_i<u_k<v_j'$. From $v_i<u_k$ it follows that $y_i<x_k$ is in $\Phi$. Hence, by reflexivity of $\sim$ we have that $x_i<x_k$ is in $\Psi$. In particular, $v_i<v_k$. Moreover, from $u_k<v_j'$ it follows that $x_k<y_j$ is in $\Phi$. Since $\Psi\subseteq\Phi$, we have $x_i<x_k\land x_k<y_j$ is in $\Phi$, a contradiction to the choice of $i$ and $j$. 
	 
	 Suppose, $\Phi$ contains an atom of the shape $y_i<x_j$. Then, by symmetry of $\sim$, $\Phi$ contains also the atom $x_i<y_j$. This case was already excluded.
	 
	 Thus, $\Phi$ contains only atoms of the shape $x_i=y_i$ for certain $i\in\{1,\dots,n\}$. 	
\end{proof}
\begin{proof}[Proof of Theorem~\ref{genposauthom}]
	It was shown by Rubin in \cite{Rub94}, that $(\mathbb{P},<)$ has a weak $\forall\exists$-interpretation. In the same paper Rubin showed that from the existence of a weak $\forall\exists$-interpretation it follows that the automorphism group has automatic homeomorphicity with respect to the class $\class{K}$ of countable $\omega$-categorical structures (cf.~\cite[Proposition 1.1.10]{BarPhD}). In particular, $\Aut(\mathbb{P},<)$ has automatic homeomorphicity with respect to $\class{K}$. Glass, McCleary and Rubin showed in  \cite[Theorem 1]{GlaMcCRub93} that $\Aut(\mathbb{P},<)$ is simple. In particular, it has a trivial center. Since $(\mathbb{P},<)$ is $\omega$-categorical, it is  saturated. Thus, from Corollary~\ref{newmonoidauthom} it follows that $\overline{\Aut(\mathbb{P},<)}$ has automatic homeomorphicity with respect to $\class{K}$.

	Let $h$ be an isomorphism from $\Pol(\mathbb{P},<)$ to the polymorphism clone of a member of $\class{K}$.

	By Theorem~\ref{genposetunivpoly}, $(\mathbb{P},<)$ has universal homogeneous polymorphisms of every arity. Thus, by Proposition~\ref{stronggateexistence}, it follows that $\Pol(\mathbb{P},<)$ has a strong gate covering. Since $\overline{\Aut(\mathbb{P},<)}$ has automatic homeomorphicity, it follows from Lemma~\ref{closed} that the restriction of $h$ to $\overline{\Aut(\mathbb{P},<)}$ is a topological embedding. In particular, $h\restr_{\overline{\Aut(\mathbb{P},<)}}$ is continuous. Since $\Pol(\mathbb{P},<)$ has a strong gate covering, it follows from Lemma~\ref{liftcont} that $h$ is continuous, too. 
		
	It remains to invoke Proposition~\ref{genposautoopen} to conclude that $h$ is a homeomorphism.
\end{proof}


\end{document}